\documentclass[graybox]{svmult}

\usepackage{mathptmx}       % selects Times Roman as basic font
\usepackage{helvet}         % selects Helvetica as sans-serif font
\usepackage{courier}        % selects Courier as typewriter font
\usepackage{type1cm}        % activate if the above 3 fonts are
\usepackage{makeidx}         % allows index generation
\usepackage{graphicx}        % standard LaTeX graphics tool
\usepackage{multicol}        % used for the two-column index
\usepackage[bottom]{footmisc}% places footnotes at page bottom

\usepackage{amsmath}
\usepackage{amssymb}
\usepackage{latexsym}
\usepackage{a4}
\usepackage{graphicx}
\usepackage[colorinlistoftodos]{todonotes}

\newcommand{\R}{{\mathbb R}}
\newcommand{\N}{{\mathbb N}}
\newcommand{\HH}{{\mathbb H}}
\newcommand{\YY}{Y}
\newcommand{\ups}{\Upsilon_{{\cal R}}^p(\R^{m\times n})}

\newcommand{\be}{\begin{eqnarray}}
\newcommand{\com}{{\beta_{{\cal R}}\R^{m\times n}}}
\newcommand{\ee}{\end{eqnarray}}
\renewcommand{\r}{\varrho}
\renewcommand{\d}{{\rm d}}
\newcommand{\cdm}{{\cal DM}^p_{\cal R}(\O;\R^{m\times n})}

\newcommand{\ra}{\right\rangle}

\newcommand{\la}{\left\langle}
\newcommand{\md}{{\rm d}}

\renewcommand{\O}{\Omega}
\newcommand{\s}{\sigma}

\newcommand{\eps}{\epsilon}
\renewcommand{\b}{\beta}

\newcommand{\A}[1]{\langle#1\rangle}
\newcommand{\Rn}{\R^{n}}

\newcommand{\ITEM}[2]{\parbox[t]{.9cm}{{\rm #1}}\hfill\parbox[t]{153mm}{#2}\vspace*{1mm}\\}

\newcommand{\rca}{{\mathcal{M}}}
\newcommand{\prca}{\rca_{\vspace*{.5mm}1}^+}
\newcommand{\wto}{\rightharpoonup}

\makeindex             % used for the subject index
\begin{document}

\title*{Weak lower semicontinuity by means of anisotropic parametrized measures}
\author{Agnieszka  Ka\l amajska, Stefan Kr\"{o}mer, and Martin Kru\v{z}\'{\i}k}
\institute{Agnieszka  Ka\l amajska \at Institute of Mathematics,
University of Warsaw, ul.~Banacha~2, PL-02-097~Warsaw, Poland, \email{Agnieszka.Kalamajska@mimuw.edu.pl} 
\and Stefan Kr\"{o}mer  \at Institute of Information Theory and Automation,
Czech Academy of Sciences, Pod vod\'{a}renskou
v\v{e}\v{z}\'{\i}~4, CZ-182~08~Praha~8, Czech Republic, \email{skroemer@utia.cas.cz} 
\and Martin Kru\v{z}\'{\i}k (corresponding author) \at Institute of Information Theory and Automation,
Czech Academy of Sciences, Pod vod\'{a}renskou
v\v{e}\v{z}\'{\i}~4, CZ-182~08~Praha~8, Czech Republic, \email{kruzik@utia.cas.cz}}
\maketitle

\abstract*{It is well known that besides oscillations, sequences bounded only in $L^1$ can also develop concentrations, and if the latter occurs, we can at most hope for weak$^*$ convergence in the sense of measures. Here we derive a new tool to handle  mutual interferences of an oscillating and concentrating sequence with another weakly converging sequence.  We introduce a couple of explicit examples showing
a variety of possible kinds of behavior  and outline  some  applications in Sobolev spaces.}

\abstract{It is well known that besides oscillations, sequences bounded only in $L^1$ can also develop concentrations, and if the latter occurs, we can at most hope for weak$^*$ convergence in the sense of measures. Here we derive a new tool to handle  mutual interferences of an oscillating and concentrating sequence with another weakly converging sequence.  We introduce a couple of explicit examples showing
a variety of possible kinds of behavior  and outline  some  applications in Sobolev spaces.}
%%%%%%%%%%%%%%%%%%%%%%%%%%%%%%%%%%%%%%%%%%%%%%%%%%%%%%%%%%%%%%%%%%%%%%%%%%%%%%%%

\section{Introduction}
Mutual interactions of oscillations and concentrations appears in many problems of optimal control and calculus of variations. We refer, for example, to
\cite{k-r-control,c-h-k} for optimal control of dynamical systems with oscillations and  concentrations, or to \cite{licht} for a model of mechanical debonding. Analytical problems related to these phenomena in the calculus of variations are described in detail in \cite{benesova-kruzik-SIREV}.
Moreover, oscillations, concentrations, and discontinuities naturally appear in problems of the variational calculus where one is interested in weak lower semicontinuity
in  the Sobolev space $W^{1,p}(\O;\R^m)$ for a sufficiently regular domain $\O\subset\R^n$ and $m,n\ge 1$. Indeed, consider
\begin{align}
I(u):=\int_\O h(x,u(x),\nabla u(x))\,\md x\ ,
\end{align}
where $h:\bar\O\times\R^m\times\R^{m\times n}\to\R$ is continuous and such that $|h(x,r,s)|\le C(1+|r|^q+|s|^p)$ for some $C>0$, $p>1$, and $q\ge 1$ so small that $W^{1,p}(\O;\R^m)$  compactly embeds into  $L^q(\O;\R^m)$.  We would like to point out that such integrands also  appear in analysis of mechanical problems \cite{paroni1,paroni2}. If one wants to investigate   lower semicontinuity of $I$ with respect to the weak topology in $W^{1,p}(\O;\R^m)$, a usual way is to show first that
\begin{align}\label{two-lims}\lim_{k\to\infty}\int_\O h(x,u(x),\nabla u_k(x))\,\md x =\lim_{k\to\infty}\int_\O h(x,u_k(x),\nabla u_k(x))\,\md x\ .
\end{align}
for a suitable sequence $u_k\wto u$ in $W^{1,p}(\O;\R^m)$, and then to prove that the left-hand side of \eqref{two-lims} is bounded from below by $\int_\O h(x,u(x),\nabla u(x))\,\md x$. That, however, is not possible without some additional assumptions on $h$ or $\{u_k\}$. We refer to \cite{af} or \cite{ball-zhang} for such cases.
Indeed, if $p\le n$ then $u$ and $u_k$, $k\in\N$,   are  not necessarily continuous and if $\{|\nabla u_k|^p\}$  is not uniformly integrable then concentrations can interact with $\{u_k\}_{k\in\N}$.  This phenomenon is clearly visible in the following example.

 \begin{example}\label{ex:intro}
Consider  $\O=B(0,1)$, the unit ball in $\R^n$ centered at the origin, a mapping  $w\in W_0^{1,p}(B(0,1);\R^m)$, $p>1$, extended by zero to the whole space and $u_k(x):=k^{n/p-1}w(kx)$. Hence $u_k\wto u:=0$ in $W^{1,p}(B(0,1);\R^m)$  as $k\to \infty$. Assume that $h$ as above is  positively $p$-homogeneous in the last variable, i.e., $h(x,r,\alpha s)=\alpha^p h(x,r,s)$, for all $(x,r,s)$ admissible  and all $\alpha\ge 0$. Then a simple calculation yields
\begin{align}
&\liminf_{k\to\infty}\int_{B(0,1)} h(x,u_k(x),\nabla u_k(x))\,\md x=\liminf_{k\to\infty}\int_{B(0,1)} k^n h(x, k^{n/p-1} w(kx),\nabla w(kx))\,\md x\nonumber\\
&=\liminf_{k\to\infty}\int_{B(0,1)} h(\frac{y}{k}, k^{n/p-1}w(y), \nabla w(y))\,\md y\nonumber \\
&= \begin{cases}
\int_{B(0,1)} h(0,w(y),\nabla w(y))\,\md y\ & \text{ if $p=n$},\\
\int_{B(0,1)} h(0,0,\nabla w(y))\,\md y\ & \text{ if $p>n$},\\
\liminf_{k\to\infty}\int_{B(0,1)} h(y/k,k^{n/p-1}w(y),\nabla w(y))\,\md y\ & \text{ if $p<n$.}\\
 \end{cases}\ \end{align}
We see that if $p>n$ then \eqref{two-lims} really holds.  On the other hand,  if $p=n$ the map $u$ appears in the limit besides its gradient and the most complex case is $p<n$ where the limit cannot be  calculated explicitly. Notice that the sequence $\{|\nabla u_k|^p\}_{k\in\N}\subset L^1(\O)$ is uniformly bounded in this space and concentrates at $x=0$, i.e., $|\nabla u_k|^p\stackrel{*}{\wto}\|\nabla u\|^p_{L^p(\O;\R^m)}\delta_0$ in $\mathcal{M}(\overline{B(0,1)})$ as $k\to\infty$. Here $\delta_0$ denotes the Dirac measure supported at the origin and $\mathcal{M}(\overline{B(0,1)})$ denotes the set of Radon measures on $\overline{B(0,1)}$.
\end{example}
If $p=1$, concentrations of the gradient can even interact with jump discontinuities.
\begin{example}\label{ex:intro2}
Consider $\O= (0,1)$ and a sequence $\{u_k\}_{k\in\N}\subset W^{1,1}(-1,1)$ such that $u_k\to u$ in $L^q(-1,1)$ for every $1\le q<+\infty$.  We are interested in
$$\lim_{k\to\infty} \int_{-1}^1  f(u_k(x))\psi(u_k'(x))\,\md x$$
for continuous function $\psi$ such that  with  $|\psi|\le C(1+|\cdot|)$ with some constant $C>0$ and continuous  $f:\R\to\R$. If $\psi$ is the identity map then the calculation  is easy, namely
the limit equals $\liminf_{k\to\infty} (F(u_k(1))-F(u_k(-1)))$ where $F$ is the primitive of $f$.
In case of more general $\psi$, the situation is more involved. Let
$$
u_k(x):=
\begin{cases}
0 &\text{ if $-1\le x\le 0$},\\
kx &\text{ if $0\le x\le 1/k$},\\
1 &\text{ if $1/k\le x\le 1$.}
\end{cases}
$$
Assume further that $\lim_{t\to\infty} \psi(t)/t$ exists. Then it is easy to see that
\begin{align}\label{limit0}
\lim_{k\to\infty} \int_{-1}^1  f(u_k(x))\psi(u_k'(x))\,\md x=(f(0)+f(1))\psi(0)+\big(\int_0^{1}f(x)\,\md x\big)\lim_{k\to\infty}\frac{\psi(k)}{k}\ .
\end{align}
The sequence of $\{u_k'\}_{k\in\N}$ concentrates at zero which is exactly the point of discontinuity of the pointwise limit of $\{u_k\}_{k\in\N}$ which we denote by $u$.
Also notice that  $u_k'\stackrel{*}{\wto}\delta_0$ in $\mathcal{M}([-1,1])$ for   $k\to\infty$. Hence, the second term on the right-hand side of \eqref{limit0} suggests that  we should refine the definition of $u$ at zero
by saying that $u(0)$ is the  Lebesgue measure supported  on the interval of the jump of $u$, i.e., on the interval $(0,1)$. 
\end{example}

In this contribution, we introduce  a new tool which allows us to describe limits of nonlinear maps along sequences that oscillate, concentrate, and concentrations possibly interfere with discontinuities. While  oscillations are successfully treated by Young measures \cite{y} or \cite{ball3}, to handle oscillations and concentrations require finer tools as in, e.g.,  Young measures and varifolds \cite{ab} or DiPerna-Majda measures \cite{diperna-majda}. We also refer to \cite{k-r-dm} for an explicit characterization of the  DiPerna-Majda measures  and to \cite{fmp, mkak} for characterization of those measures which are generated by sequences of gradients, as well as to 
\cite{KriRin_YM_10} and \cite{BKK16} for related results in case $p=1$.

\subsection{Basic notation}\label{bn}
Let us  start with a few definitions and with the explanation of
our notation. If not said otherwise, we will  assume throughout this article that
$\O\subset\R^n$ is a bounded domain with a Lipschitz boundary.
Furthermore, $C(\O;\R^m)$ (respectively $C(\bar{\O};\R^m)$) is the space of continuous functions defined on $\O$ (respectively $\
bar{\O}$) with values in $\R^m$. Here, as well as in smilar notation for other function spaces, if the dimension of the target space is $m=1$, then $\R^m$ is omitted and we only write $C(\O)$.
In what follows $\mathcal{M}(S)$ denotes the set of regular countably
additive set functions on the Borel $\s$-algebra on a metrizable
set  $S$ (cf. \cite{d-s}), its subset, $\mathcal{M}^+_1(S)$, denotes
regular probability measures on a set $S$. We write
``$\gamma$-almost all'' or ``$\gamma$-a.e.'' if  we mean ``up to a
set with the $\gamma$-measure zero''. If $\gamma$ is the
$n$-dimensional Lebesgue measure  we omit writing $\gamma$ in the
notation. The support of a measure $\sigma\in\mathcal{M}(\O)$ is the
smallest closed set $S$ such that $\sigma(A)=0$ if $S\cap
A=\emptyset$. If $\sigma\in\mathcal{M}(\bar\O)$ we write $\sigma_s$ and
$d_\sigma$ for the singular part and density  of $\sigma$ defined
by   the Lebesgue decomposition (with respect to the Lebesgue
measure), respectively.
By $L^p(\O;\R^m)$ we denote the usual Lebesgue space of $\R^m$-valued maps.
Further, $W^{1,p}(\O;\R^m)$ where $1\le p\le+\infty$ denotes the
usual Sobolev space (of $\R^m$-valued functions) and
$W_0^{1,p}(\O;\R^m)$ denotes the completion of $C_0^\infty
(\O,\R^m)$ (smooth functions with support in $\O$) in
$W^{1,p}(\O;\R^m)$. We say that $\Omega$ has the
extension property in $W^{1,p}$ if every function $u\in
W^{1,p}(\Omega)$ can be extended outside $\Omega$ to $\tilde{u}\in
W^{1,p}(\R^n)$ and the extension operator is linear and bounded.
If $\O$ is an arbitrary domain and $u,w\in W^{1,p}(\O,\R^m)$ we
say that $u=w$ on $\partial\O$ if $u-w\in W_0^{1,p}(\O;\R^m)$.
 We denote by `w-$\lim$'  or by $\rightharpoonup$ the weak limit. Analogously we indicate weak* limits by $\stackrel{*}{\wto}$.

%\bigskip

%\bigskip

\bigskip

\subsection{Quasiconvex functions}

Let $\O\subset\R^n$ be a bounded domain. 
We say that a function $\psi:\R^{m\times n}\to\R$ is quasiconvex if
for any $s_0\in\R^{m\times n}$ and any $\varphi\in W^{1,\infty}_0(\O;\R^m)$
$$
\psi(s_0)|\O|\le \int_\O \psi(s_0+\nabla \varphi(x))\,\md x\ .$$ If
$\psi:\R^{m\times n}\to\R$ is not quasiconvex we define its
quasiconvex envelope $Q\psi:\R^{m\times n}\to\R$ as \be\label{qcdef}
Q\psi(s)=\sup\left\{h(s);\ h\le \psi;\ \mbox{$h:\R^{m\times n}\to\R$
quasiconvex }\right\}\ \ee and we put $Q\psi=-\infty$ if the set on
the right-hand side of (\ref{qcdef}) is empty. If $\psi$ is locally
bounded and Borel measurable then for any $s_0\in\R^{m\times n}$
(see \cite{dacorogna}) \be\label{relaxation}
Q\psi(s_0)=\inf_{\varphi\in W^{1,\infty}_0(\O;\R^m)} \frac{1}{|\O|}
\int_\O \psi(s_0+\nabla \varphi(x))\,\md x\ .\ee 
%If $|\psi(s)|\le
%C(1+|s|^p)$ for some $C>0$ and all $s\in\R^{m\times n}$  then
%equivalently
%$$Q\psi(s_0)=\inf_{\varphi\in W^{1,p}_0(\O;\R^m)} \frac{1}{|\O|} \int_\O \psi(s_0+\nabla \varphi(x))\,\md x\ ,$$
%as pointed out   in \cite{fmp}. We refer to \cite{murat} for the notion of $W^{1,p}$-quasiconvexity.
%
%Let us point out that
%$$Q\psi(s_0)=\inf_{\varphi\in W^{1,p}_{s_0}(\O;\R^m)} \frac{1}{|\O|} \int_\O \psi(\nabla \varphi(x))\,\md x\ ,$$
%where $W^{1,p}_{s_0}(\O;\R^m)=\{\varphi\in W^{1,p}(\O;\R^m);\ \varphi(x)=s_0x \mbox{ on $\partial\O$ }\}$.

%%% HE

%We will also need the following  elementary  result. It can be found in a more general form   e.g. in \cite[Ch.~4, Lemma~2.2]{dacorogna}  or in \cite{morrey}.
%\beginlemma\labellemma  Let $v:\R^{m\times n}\to\R$ be quasiconvex  with  $|v(s)|\le C(1+|s|^p)$, $C>0$, for all $s\in\R^{m\times n}$.
%Then there is a constant $\alpha\ge 0$ such that  for every  $s_1,s_2\in\R^{m\times n}$ it holds
%\be
%|v(s_1)-v(s_2)|\le \alpha(1+|s_1|^{p-1}+ |s_2|^{p-1})|s_1-s_2|\ .\ee
%\endlemma

\subsection{Young measures}
For $p\ge0$ we define the following  subspace of the space
$C(\R^{m\times n})$ of all continuous functions on $\R^{m\times n}$ :
$$
C_p(\R^{m\times n})=\{\psi\in C(\R^{m\times n}); \psi(s)=o(|s|^p)\mbox{ for
}|s|\rightarrow\infty\}\ ,
$$
with the obvious modification for any Euclidean space instead of $\R^{m\times n}$.
The Young
measures on a measurable set $\Lambda\subset \R^l$ 
%bounded  domain $\O\subset\Rn$ 
are weakly* measurable mappings
$x\mapsto\nu_x:\Lambda\to \rca(\R^{m\times n})$ with values in probability measures;
 and the adjective ``weakly* measurable'' means that,
for any $\psi\in C_0(\R^{m\times n})$, the mapping
$\Lambda\to\R:x\mapsto\A{\nu_x,\psi}=\int_{\R^{m\times n}}
\psi(s)\nu_x(\md s)$ is measurable in the usual sense. Let
us remind that, by the Riesz theorem the space $\rca(\R^{m\times
n})$, normed by the total variation, is a Banach space which is
isometrically isomorphic with $C_0(\R^{m\times n})^*$.  Let us
denote the set of all Young measures by ${\cal Y}(\Lambda;\R^{m\times
n})$. 

Below, we are mostly interested in the case $\Lambda=\O$, i.e., a bounded domain.
It is known that ${\cal Y}(\O;\R^{m\times n})$ is a convex
subset of $L^\infty_{\rm w*}(\O;\rca(\R^{m\times n}))\cong
L^1(\O;C_0(\R^{m\times n}))^*$, where the index ``$w*$''
indicates the property ``weakly* measurable''.  A classical result
\cite{y} is that, for every sequence $\{y_k\}_{k\in\N}$
bounded in $L^\infty(\O;\R^{m\times n})$, there exists its
subsequence (denoted by the same indices for notational
simplicity) and a Young measure $\nu=\{\nu_x\}_{x\in\O}\in{\cal
Y}(\O;\R^{m\times n})$ such that \be\label{jedna2} \forall \psi\in
C_0(\R^{m\times n}):\ \ \ \ \lim_{k\to\infty}\psi\circ y_k=\psi_\nu\ \ \
\ \ \ \mbox{ weakly* in }L^\infty(\O)\ , \ee where $[\psi\circ
y_k](x)=\psi(y_k(x))$ and \be \psi_\nu(x)=\int_{\R^{m\times
n}}\psi(s)\nu_x(\d s)\ . \ee Let us denote by ${\cal
Y}^\infty(\O;\R^{m\times n})$ the set of all Young measures which
are created by this way, i.e. by taking all bounded sequences in
$L^\infty(\O;\R^{m\times n})$. Note that (\ref{jedna2}) actually
holds for any $\psi:\R^{m\times n}\to\R$ continuous.

A generalization of this result was formulated by
Schonbek \cite{schonbek} (cf. also \cite{ball3}): if
$1\le p<+\infty$: for every sequence
$\{y_k\}_{k\in\N}$ bounded in $L^p(\O;\R^{m\times n})$ there exists its
subsequence (denoted by the same
indices) and a Young measure
$\nu=\{\nu_x\}_{x\in\O}\in{\cal Y}(\O;\R^{m\times n})$ such that
\be\label{young}
\forall \psi\in C_p(\R^{m\times n}):\ \ \ \ \lim_{k\to\infty}\psi\circ y_k=\psi_\nu\
\ \ \ \ \ \mbox{ weakly in }L^1(\O)\ .\ee
We say that $\{y_k\}$ generates $\nu$ if \eqref{young} holds.
 Let us denote by ${\cal
Y}^p(\O;\R^{m\times n})$ the set of all Young measures which are created by this
way, i.e. by taking all bounded sequences in $L^p(\O;\R^{m\times n})$.
The subset of ${\cal Y}^p(\O;\R^{m\times n})$ containing Young measures
generated by gradients of $W^{1,p}(\O;\R^m)$ maps will be denoted by
 ${\cal GY}^p(\O;\R^{m\times n})$. An explicit characterization of this set is due to Kinderlehrer and Pedregal \cite{k-p,k-p1}.

%We will use the following lemma from \cite{fmp} concerning Young measures
%from ${\cal Y}^p(\O;\R^{m\times n})$ which are generated by sequences of gradients. A similar result was also proved by Kristensen \cite{thesis}.
%
%\bigskip
%
%
%\begin{lemma}\label{fons}
%Let $1 < p<+\infty$ and  $\O\subset\R^n$ be an open bounded set and let $\{ u_k\}_{k\in\N}\subset W^{1,p}(\O;\R^m)$ be bounded. Then there is a subsequence $\{u_j\}_{j\in\N}$ and a sequence $\{z_j\}_{j\in\N}\subset W^{1,p}(\O;\R^m)$ such that
%\be\label{rk}
%\lim_{j\to\infty} \left|\{x\in\O;\ z_j(x)\ne u_j(x)\mbox{ or }  \nabla z_j(x)\ne \nabla u_j(x)\}\right|=0
%\ee
%and $\{|\nabla z_j|^p\}_{j\in\N}$ is relatively weakly compact in $L^1(\O)$.
%In particular, $\{\nabla u_j\}$ and $\{\nabla z_j\}$ generate the same Young measure.
%\end{lemma}

\bigskip

\subsection{DiPerna-Majda measures}

\subsubsection{Definition and basic properties}

 Let ${\cal R}$ be a complete
(i.e. containing constants, separating points from closed subsets
and closed with respect to the Chebyshev norm) separable  ring  of
continuous bounded functions $\R^{m\times n}\to\R$. It is known
\cite[Sect.~3.12.21]{engelking} that there is a one-to-one
correspondence ${\cal R}\leftrightarrow\b_{\cal R}\R^{m\times n}$
between such rings and metrizable compactifications of
$\R^{m\times n}$; by a compactification we mean here a compact
set, denoted by $\b_{\cal R}\R^{m\times n}$, into which
$\R^{m\times n}$ is embedded homeomorphically and densely. For
simplicity, we will not distinguish between $\R^{m\times n}$ and
its image in $\b_{\cal R}\R^{m\times n}$. Similarly, we will not
distinguish between elements of ${\cal R}$ and their unique
continuous extensions defined  on $\b_{\cal R}\R^{m\times n}$.
This means that if $i: \R^{m\times n}\rightarrow \b_{\cal
R}\R^{m\times n}$ is the homeomorphic embedding and $\psi_0\in {\cal
R}$ then the same notation is used also for $\psi_0\circ i^{-1}:
i(\R^{m\times n})\rightarrow \R$ and for its unique continuous
extension to $\com$.

 Let $\s\in\rca(\bar\O)$ be a  positive Radon measure on a closure of a bounded domain $\O\subset\R^n$. A
mapping $\hat\nu:x\mapsto \hat\nu_x$ belongs to the
space $L^{\infty}_{\rm w*}(\bar{\O},\s;\rca(\b_{\cal R} \R^{m\times n}))$ if it is weakly*  $\s$-measurable (i.e., for any $\psi_0\in C_0(\R^{m\times n})$, the mapping
$\bar\O\to\R:x\mapsto\int_{\b_{\cal R}\R^{m\times n}} \psi_0(s)\hat\nu_x(\d s)$ is $\s$-measurable in
the usual sense). If additionally
$\hat\nu_x\in\prca(\b_{\cal R}\R^{m\times n})$ for $\s$-a.a. $x\in\bar\O$
 the collection $\{\hat\nu_x\}_{x\in\bar{\O}}$ is the so-called
Young measure on $(\bar\O,\s)$ (\cite{y}, see also
\cite{ball3,r}).

DiPerna and Majda \cite{diperna-majda} shown that having a bounded
sequence in $L^p(\O;\R^{m\times n})$ with $1\le p<+\infty$ defined
on an open domain $\O\subseteq\Rn$, there exists its subsequence
(denoted by the same indices) a positive Radon measure
$\s\in\rca(\bar\O)$ and a Young measure  $\hat\nu:x\mapsto
\hat\nu_x$  on $(\bar\O,\s)$ such that  $(\s,\hat\nu)$ is
attainable by a sequence $\{y_k\}_{k\in\N}\subset
L^p(\O;\R^{m\times n})$ in the sense that $\forall g\!\in\!
C(\bar\O)\ \text{ and } \forall \psi_0\!\in\!{\cal R}$:
\be\label{basic}\lim_{k\to\infty}\int_\O g(x)\psi(y_k(x))\d x =
\int_{\bar\O}g(x)\int_{\b_{\cal R}\R^{m\times
n}}\psi_0(s)\hat\nu_x(\d s)\s(\d x)\ , \ee where \be\label{upes}
\psi\in\ups:=\{\psi_0(1+|\cdot|^p);\ \psi_0\in{\cal R}\}.\ee
 In particular,
putting $\psi_0\equiv 1\in{\cal R}$ in (\ref{basic}) we can see that
\be\label{measure} \lim_{k\to\infty}(1+|y_k|^p)\ =\ \s \ \ \ \
\mbox{ weakly* in }\ \rca(\bar\O)\ . \ee If (\ref{basic}) holds,
we say that $\{y_k\}_{\in\N}$ generates $(\sigma,\hat\nu)$. Let us
denote by ${\cal DM}^p_{\cal R}(\O;\R^{m\times n})$ the set of all
pairs $(\s,\hat\nu)\in\rca(\bar\O)\times L^{\infty}_{\rm
w*}(\bar{\O},\s; \rca(\b_{\cal R} \R^{m\times n}))$ attainable by
sequences from $L^p(\O;\R^{m\times n})$; note that, taking $\psi_0=1$
in (\ref{basic}), one can see that these sequences must be
inevitably bounded in $L^p(\O;\R^{m\times n})$.

It is well known \cite{r} that \eqref{basic} can also be rewritten with the help of classical Young measures as
\begin{align}\label{y-dm}
\lim_{k\to\infty}\int_\O g(x)\psi(y_k(x))\d x& =
\int_\O\int_{\R^{m\times n}} g(x)\psi(s)\nu_x(\md s) \d x\nonumber\\
&+ \int_{\bar\O}g(x)\int_{\b_{\cal R}\R^{m\times
n}\setminus\R^{m\times n}}\psi_0(s)\hat\nu_x(\d s)\s(\d x),\
\end{align}
where $\{ \nu_x\}_{x\in \O} \in {\cal Y}^\infty (\O,\R^{m\times n})$ and  $\{ \nu_x\}_{x\in \O}$ are as in (\ref{basic}).

There are two prominent examples of compactifications of $\R^{m\times n}$. 
The simplest example is the so-called one point compactification which corresponds to the ring of continuous bounded functions which have limits if the norm of its argument tends to infinity, i.e., we denote $\psi_0(\infty):=\lim_{|s|\to+\infty}\psi_0(s)$.

A richer compactification is the one by the sphere. In that case, we consider
the following ring of continuous bounded functions:
\begin{align}\label{spherecomp}
\mathcal{S}:=\big\{ &\psi_0\in C(\R^{m\times n}):\mbox{ there exist } c\in\R\ ,\ \psi_{0,0}\in C_0(\R^{m\times n}),\mbox{ and }  \psi_{0,1}\in C(S^{(m\times n)-1}) \mbox{ s.t. }\nonumber\\
&  \psi_0(s) = c+ \psi_{0,0}(s)+\psi_{0,1}\left(\frac{s}{|s|}\right)
\frac{|s|^p}{1+|s|^p}\mbox { if $s\ne 0$ and }  \psi_0(0)=\psi_{0,0}(0)\big\}\ ,
\end{align}
where $S^{m\times n-1}$ denotes the $(mn-1)$-dimensional unit sphere in $\R^{m\times n}$. Then $\b_{\cal
R}\R^{m\times n}$ is homeomorphic to the unit ball $\overline{B(0,1)}\subset \R^{m\times n}$ via the mapping $d:\R^{m\times n}\to B(0,1)$, $d(s):=s/(1+|s|)$ for all $s\in\R^{m\times n}$. Note that $d(\R^{m\times n})$ is dense in $\overline{B(0,1)}$.

%For every $\psi=\psi_0(1+|\cdot|^p)$ with $\psi_0\in\mathcal{S}$ from \eqref{spherecomp}   there exists a continuous and positively $p$-homogeneous function $\psi_\infty:\R^{m\times n}\to\R$ (i.e. $\psi_\infty(\alpha s)=\alpha^p \psi_\infty(s)$ for all $\alpha\ge 0 $ and $ s\in\R^{m\times n}$) such that
%\be\label{recessionf}
%\lim_{|s|\to\infty}\frac{\psi(s)-\psi_\infty(s)}{|s|^p}=0\ .
%\ee
%
%Indeed, if $\psi_0$ is as in (\ref{spherecomp}) and $\psi=\psi_0(1+|\cdot|^p)$  then set
%$$\psi_\infty(s):=\left(c+\psi_{0,1}\left(\frac{s}{|s|}\right)\right)|s|^p\mbox{ for $s\in\R^{m\times n}\setminus\{0\}$.} $$
%By continuity we define $\psi_\infty(0):=0$. It is easy to see that $\psi_\infty$ satisfies (\ref{recessionf}).
%Such $\psi_\infty$ is called the  recession function of $\psi$. It is clear that the same construction works for the Euclidean space of every dimension.

\bigskip

 The following proposition from \cite{k-r-dm} explicitly characterizes
the set of DiPerna-Majda measures $\cdm$.

\bigskip

\begin{proposition}\label{characterization}
Let $\O\subset\R^n$ be a bounded open domain such that
$|\partial\O | =0$, ${\cal R}$ be a separable complete subring of
the ring of all continuous bounded functions on $\R^{m\times n}$
and $(\s,\hat{\nu})\in \rca(\bar{\O})\times L^{\infty}_{\rm
w}(\bar{\O},\s; \rca(\b_{\cal R}\R^{m\times n}))$ and $1\le
p<+\infty$.
Then the following two statements are equivalent with each other:\\
\ITEM{(i)}{the pair $(\s,\hat\nu)$ is the DiPerna-Majda measure, i.e.
$(\s,\hat\nu)\in{\cal DM}^p_{\cal R}(\O;\R^{m\times n})$,}
\ITEM{(ii)}{The following properties are satisfied simultaneously:
\begin{enumerate}
\item
$\s$ is positive,
\item
$\s_{\hat\nu}\in\rca(\bar\O)$ defined by
$\s_{\hat\nu}(\d x)=(\int_{\R^{m\times n}}\hat\nu_x(\d s))\s(\d x)$
is
absolutely\\
 continuous with respect to the Lebesgue
measure\\
 ($d_{\s_{\hat\nu}}$ will denote its density),
\item for a.a. $x\in\O$ it holds
$$
\!\!\!\!\!\!\!\!\!\!\!\!\!\!\!\!\!\!\!\!\!\!\!\!\!\!\!\!\!\!\!\!\!
\!\!\!\!\!\!\!\!\!\!\!\!\!\!\!\!\!\!\!\!\!\!\!\!\!\ \int_{\R^{m\times n}}\hat\nu_x(\d s) >0,\ \ \ \ \ \ d_{\s_{\hat\nu}}(x)
=\left(\int_{\R^{m\times n}}\frac{\hat\nu_x(\d s)}
{1+|s|^p}\right)^{-1}\int_{\R^{m\times n}}\hat\nu_x(\d s)\ ,
$$
\item
for $\s$-a.a. $x\in\bar\O$ it holds
$$
\hat\nu_x\ge 0,\ \ \ \ \ \
\int_{\b_{\cal R}\R^{m\times n}}\hat\nu_x(\d s)=1\ .
$$
\end{enumerate}
}
\end{proposition}

\begin{remark}\label{finercompactification}
 Consider a metrizable compactification $\beta_{\cal R}\R^{m\times n}$ of $\R^{m\times n}$ and the corresponding
separable complete closed  ring ${\cal R}$ with its dense subset $\{\psi_k\}_{k\in\N}$.  We  take a bounded continuous  function $\psi:\R^{m\times n}\to\R$, $\psi\not\in{\cal R}$ and take a closure (in the Chebyshev norm) of all the products of elements from  $\{\psi\}\cup\{\psi_k\}_{k\in\N}$.  The corresponding ring is again separable and  the corresponding compactification is metrizable but strictly  finer than  $\beta_{\cal R}\R^{m\times n}$.
\end{remark}

The following result can be found in \cite{mkak} and its extension in \cite{kroemer-kruzik}. Here and in the sequel $d_\sigma$ denotes density of the absolutely continuous part of $\sigma$ with respect to the Lebesgue measure  $\mathcal{L}^n$.
\begin{theorem}\label{suff1}
Let  $\O\subset\R^n$ be a bounded   domain with the extension property in $W^{1,p}$,
$1<p<+\infty$ and $(\sigma,\hat\nu)\in\cdm$. Then then there is a bounded sequence
$\{u_k\}_{k\in\N}\subset W^{1,p}(\O;\R^m)$ such that $u_k=u_j$ on $\partial\O$ for any $j,k\in\N$ and  $\{\nabla
u_k\}_{k\in\N}$ generates $(\sigma,\hat\nu)$ if and only if the
following three conditions hold:
\be\label{firstmoment6} 
\exists u\in W^{1,p}(\O;\R^m):\mbox{  for a.a. $x\in\O$:  }  \nabla
u(x)=d_\sigma(x)\int_{\beta_{\cal R}\R^{m\times n}}\frac{s}{1+|s|^p}\hat\nu_x(\d
s)\ ,
\ee 
for almost all $x\in\O$ and for all  $\psi_0\in\R$ and $\psi(s):=(1+|s|^p)\psi_0(s)$,  the
\be\label{qc6} Q\psi(\nabla
u(x))\le d_\sigma(x)\int_{\beta_{\cal R}\R^{m\times
n}}\psi_0(s)\hat\nu_x(\d s)\ , \ee for $\sigma$-almost
all $x\in\bar\O$ and all $\psi_0\in\R$ with $Q\psi>-\infty$, where $\psi(s):=(1+|s|^p)\psi_0(s)$,
\be\label{rem6}
 0\le  \int_{\beta_{{\cal R}}\R^{m\times n}\setminus\R^{m\times n}}\psi_0(s)\hat\nu_x(\md s)\ .
\ee
\end{theorem}
\begin{remark}
Inequality \eqref{qc6} can be 
%\footnote{AK: equivalently -erased. I think it is not true -- SK: Due to \eqref{rem6}, $\beta_\mathcal{R}$ in  \eqref{qc6} is redundant and could be removed (use finite part only). And then equivalence does hold.} 
written in terms of $\nu=\{\nu_x\}$, the Young measure generated by $\{u_k\}$, as
follows \cite{k-p1}: There exists a zero-measure set $\omega\subset\O$  such that for every $x\in\O\setminus\omega$
\begin{align}
 \psi(\nabla
u(x))\le \int_{\R^{m\times
n}}\psi(s)\nu_x(\d s)\ ,
\end{align}
for all $\psi:\R^{m\times n}\to\R$ quasiconvex and such that $|\psi|\le C(1+|\cdot|^p)$ for some $C>0$.
\end{remark}
Theorem~\ref{suff1} can be used to obtain weak lower semicontinuity results along sequences with prescribed boundary data \cite{mkak}. If we do not control boundary conditions the situation is much more subtle.  To the best of our knowledge, the first results in this direction are due to Meyers \cite{meyers} who also deals with higher-order variational problems. However, his condition is stated in terms of sequences. 
A refinement was proved in \cite[Thm.~1.6]{kroemer}, showing that even near the boundary, the necessary and sufficient conditions 
for weak lower semicontinuity in terms of 
the integrand can be expressed in terms of localized test functions, similar to quasiconvexity:
%

%\begin{thm}\label{thm:GDMchar}
%Assume that \eqref{H1}--\eqref{H3} hold,
%and let $(\sigma,\hat\nu)\in \DM$. 
%Then $(\sigma,\hat\nu) \in \GDM$
%%, i.e., $(\sigma,\hat\nu)$ is generated by a sequence of gradients, 
%%there is a bounded sequence $(u_k)\subset W^{1,p}(\O;\RR^M)$ such that $(\nabla u_k)$ generates $(\sigma,\hat\nu)$ 
%if and only if the following four conditions are satisfied simultaneously:
%\begin{enumerate}
%
%\item[(i)] There exists $u\in W^{1,p}(\O;\RR^M)$ such that for a.e.~$x\in\O$,
%\begin{equation*}%\label{GDMc-1}
	%\nabla u(x)=d_\sigma(x)\int_{\com} \frac{s}{1+\abs{s}^p}\hat{\nu}_x(ds);
%\end{equation*}
%
%\item[(ii)] With $u$ from (i), for a.e.~$x\in\O$ and every $v\in \ups$,
%\begin{equation*}%\label{GDMc-2}
	%Qv(\nabla u(x))\leq d_\sigma(x)\int_{\com} \frac{v(s)}{1+\abs{s}^p}\hat{\nu}_x(ds);
%\end{equation*}
%
%\item[(iii)] For $\sigma$-a.e.~$x\in \O$ and every $v\in \ups$ such that $Q v>-\infty$,
%\begin{equation*} %\label{GDMc-3i}
	%0\leq \int_{\com\setminus \RR^{M\times N}} \frac{v(s)}{1+\abs{s}^p}\hat{\nu}_x(ds);
%\end{equation*} 
%
%\item[(iv)] For $\sigma$-a.e.~$x\in \partial\O$ and every $v\in \ups$ which is $p$-qscb at $x$,
%\begin{equation*} %\label{GDMc-3b}
	%0\leq \int_{\com\setminus \RR^{M\times N}} \frac{v(s)}{1+\abs{s}^p}\hat{\nu}_x(ds).
%\end{equation*} 
%
%\end{enumerate}
%Here, $d_\sigma$ denotes the density of the absolutely continuous part of $\sigma$ with respect to the Lebesgue measure,
%which is explicitly given by \eqref{dsigma}.
%\end{thm}

\begin{theorem}\label{kroemer}
Let $1<p<\infty$,  $\O\subset \R^n$ be a bounded domain with  the $C^1$-boundary. Let  $\tilde h:\bar\O\times\R^{m\times n}\to\R$ be continuous and such that $ \tilde h(\cdot,s)/(1+|s|^p)$ is bounded and continuous in $\bar\O$, uniformly in $s$.  Then $J(u):=\int_\O \tilde h(x,\nabla u(x))\,\md x$  is weakly lower semicontinuous in $W^{1,p}(\O;\R^m)$ if and only if the following two conditions hold simultaneously:
\vspace*{-1ex}
\begin{alignat}{2}
& \text{(i)} &&
\text{$\tilde h(x,\cdot)$ is quasiconvex for all $x\in\O$;}\nonumber\\
& (ii)~ && \text{for every $x_0\in\partial\O$ and for every $\eps>0$, there exists $C_\eps\geq 0$ such that}\nonumber\\
&&& \label{pqslb}
\int_{D_\varrho} \tilde h(x_0,\nabla \varphi(x))\,\d x\geq -\eps \int_{D_\varrho} |\nabla \varphi(x)|^p\,\d x-C_\eps~~
	\text{for every $\varphi \in C^\infty_c(B(0,1);\R^m)$}.
%&\text{for every $\varphi \in W^{1,p}(D_\nu;\R^M)$ with $\varphi=0$ on $\Gamma_\nu$}. %in the sense of trace
\end{alignat}
Here, $D_\varrho:=\{x\in B(0,1);\ x\cdot\varrho<0\}$ where $\varrho$ denotes the outer unit normal to $\partial\O$ at $x_0$.
\end{theorem}
If $\tilde h$ satisfies (ii) we say that it has $p$-quasisubcritical growth from below ($p$-qscb) at $x_0$.

\section{Anisotropic parametrized measures generated by pairs of sequences}
This section is devoted to a new tools which might be seen as a multiscale oscillation/concentration measures. It is a generalization of the approach introduced in \cite{pedregal-mult} where only oscillations were taken into account.  We also wish to mention that if $\{u_k\}_{k\in\N}$ is bounded in $W^{1,p}(\O;\R^m)$ for $1<p<\infty$ then (at least for a nonrelabeled subsequence) the Young measure generated by the pair $\{(u_k,\nabla u_k)\}$ is $\xi_x(\md(r,s))=\delta_{u(x)}(\md r)\nu_x(\md s)$ for almost all $x\in\O$.
Here $u$ is the weak limit of $\{u_k\}_{k\in\N}$ in $W^{1,p}(\O;\R^m)$ and $\{\nu_x\}_{x\in\O}$ is the Young measure generated by $\{\nabla u_k\}$. We refer to \cite{pedregal} for the proof of this statement. If we are interested also in concentrations of $\{|\nabla u_k|^p\}$ and in their interactions with $\{u_k\}$ the situation is more involved.

As before, let $\mathcal{R}$ be a complete separable  ring  of
continuous bounded functions $\R^{m\times n}\to\R$. %, containing at least $C_0(\R^{m\times n})$ and all constant functions. 
Similarly, we take a complete separable ring $\mathcal{U}$ of continuous bounded real-valued  functions on $\R^m$, 
%containing at least $C_0(\R^m)$ and all constant functions, 
and denote the corresponding metrizable compactification of $\R^m$ by $\b_{\mathcal{U}}\R^m$. 
We will consider the ring $C(\bar\O)\otimes \mathcal{U}\otimes \mathcal{R}$, the  subset
 of bounded continuous functions on $\O\times \R^{m}\times \R^{m\times n}$ spanned by $\{ (x,s,r)\mapsto g(x)f_0(r)\psi_0(s): g\in C(\bar\O),~~f_0\in \mathcal{U},~~ \psi_0\in \mathcal{R}\}$. Also notice that $\b_\mathcal{U}\R^{m}\times \b_\mathcal{R}\R^{m\times n} =\b_{\mathcal{U}\otimes \mathcal{R}}(\R^{m}\times \R^{m\times n})$.  
Finally, notice that the linear hull of $\{g\otimes f_0\otimes \psi_0:\, g\in C(\bar\O)\, , f_0\in C(\b_\mathcal{U})\, ,\psi_0\in C(\b_\mathcal{R}\R^{m\times n})\}$ is dense in $C(\bar\O\times\b_\mathcal{U}\R^{m}\times \b_\mathcal{R}\R^{m\times n})$ due to the Stone-Weierstrass theorem.
Here, $[g\otimes f_0\otimes \psi_0](x,r,s):=g(x)f_0(r)\psi_0(s)$ for all $x\in\bar\O$, $r\in\R^m$, and all $s\in\R^{m\times n}$.

\begin{remark}
There always exists a separable ring 
into which a given continuous bounded function $f_0$ belongs. Indeed, consider a ring $\mathcal{U}_0$ of continuous functions which possess limits  if the norm of their argument tends to infinity. This ring to the one-point compactification of $\R^m$. If $f_0$ does not belong to $\mathcal{U}_0$ we construct a larger ring from $f_0$ and $\mathcal{U}$ by taking the closure (in the maximum norm) of all products of $\{f_0\}\cup\mathcal{U}$.  
\end{remark}

\subsection{Representation of limits using parametrized measures}

The following statement is rather standard generalization of the DiPerna-Majda Theorem to the anisotropic case. It can be obtained using a special case of the representation theorem in \cite{ak2}.\footnote{in \cite{ak2} it is assumed that the compactification of the entire space $\R^m\times\R^{m\times n}$ is a subset in $\R^N$ for some $N\in\N$. This however is not required for the proof in \cite{ak2} which only uses separability of the compactification.}
%which deals with the case of one ``brick'' in the decomposition of the entire space and considers ``density function'' $g(r,s) = 1+|r|^q+|s|^p$.
\begin{theorem}\label{thm0}
Let  $1\le q\le +\infty$, $1\le p<+\infty$ and
\[
\YY^{q,p}(\O,\mathcal{U},\mathcal{R})=\{ h_0(r,s)(1+|r|^q+|s|^p) : h_0\in C(\bar\O\times \b_\mathcal{U}\R^{m}\times \b_\mathcal{R}\R^{m\times n})\} .
\]
Moreover, let  $\{u_k\}_{k\in\N}$ be  bounded sequence in  $L^q(\O;\R^m)$ and $\{w_k\}$ a bounded sequence in $L^p(\O;\R^{m\times n})$. Then there is a  subsequence $\{(u_k,w_k)\}$ (denoted by the same indeces), a measure  $\hat{\sigma}(dx)$
%with the following properties:
such that
%\begin{description}
%\item[i)] 
$$(1+|u_k|^q + |w_k|^p)dx\stackrel{*}{\rightharpoonup} \hat{\sigma},$$ 
%weakly $*$ converges to some measure denoted by $\hat{\sigma}(dx)$;f
%\item[ii)] $\{(u_k,w_k)\}$ generates classical Young measures $\{\xi_x\}_{x\in\O} \in {\cal Y}^\infty
%(\O\times\R^m\times\R^{m\times n})$, i.e. for every $f\in C_0(\R^m\times\R^{m\times n})$ and every $g\in L^\infty (\O)$ we have 
%\[
%\lim_{k\to\infty}\int_\O g(x) f(u_k(x),w_k(x))dx = \int_\O \int_{ \R^m\times\R^{m\times n} } f(r,s)\xi_x(dr,ds) g(x)dx;
%\]
%\item[iii)] 
and a family of probability measures $\{ \hat{\gamma}_x\}_{x\in\bar{\Omega}}\in L^\infty_{{\rm w}*}(\bar{\Omega}, \mathcal{M}(\b_{\mathcal{U}}\R^m \times \b_{\mathcal{R}}\R^{m\times n});\hat{\sigma})$ such that
for any $h\in \YY^{q,p}(\O,\mathcal{U},\mathcal{R})$ and any $g\in C(\bar{\O})$ we have
\begin{eqnarray*}
\lim_{k\to\infty} \int_\O g(x)h_0(u_k(x),w_k(x))(1+|u_k(x)|^q+|w_k(x)|^p) dx \rightarrow\\
\int_{\bar{\O}} g(x)\int_{\b_\mathcal{U}\R^{m}\times \b_\mathcal{R}\R^{m\times n}} h_0(r,s)\hat{\gamma}_x(dr,ds) \hat{\sigma}(\d x).
\end{eqnarray*}
%\item[iv)] for any $h\in \YY^{q,p}(\O,\mathcal{U},\mathcal{R})$ and any $g\in C(\bar{\O})$ we have
%\begin{eqnarray*}
%\int_{\bar{\O}} g(x)\int_{\b_\mathcal{U}\R^{m}\times \b_\mathcal{R}\R^{m\times n}} h_0(r,s)%\hat{\gamma}_x(dr,ds) \hat{\sigma}(\d x) =\\ \int_{{\O}} g(x) \int_{ \R^m\times\R^{m\times n} }
%h(r,s) \xi_x(dr,ds) dx + \int_{\bar{\O}} g(x)\int_{Rem    } h_0(r,s)\hat{\gamma}_x(dr,ds) \hat{\sigma}(\d x)
%\end{eqnarray*}
%where $Rem = \b_\mathcal{U}\R^{m}\times \b_\mathcal{R}\R^{m\times n}\setminus \R^m\times \R^{m\times n}$ is the reminder.
%\end{description}
\end{theorem}
\begin{remark} \label{rem:catch}
In a sense, the pair $(\hat{\sigma},\hat{\gamma})$ is an anisotropic $(q,p)$ DiPerna-Majda measure generated by the sequence $\{ (u_k,w_k)\}$,
generalizing the isotropic case $p=q$.
However, while this approach is a rather intuitive generalization of standard DiPerna-Majda measures, it has a drawback: Several extremely simple and often prototypical choices for the integrands which we would like to use in applications are not admissible.
For instance, $h(x,r,s):=|s|^p$ \emph{never} is an element of $\YY^{q,p}(\O,\mathcal{U},\mathcal{R})$, because 
the limit of $h_0(x,r,s):=|s|^p(1+ |r|^q+|s|^p)^{-1}$ as $|(r,s)|\to \infty$ does not exist: we get $1$ as $|s|\to \infty$ for fixed $r$, and $0$ as $|s|\to \infty$ for fixed $r$. Hence, this function $h_0$ does not have a continuous extension to the compactification $\beta_{\mathcal{U}}\times \beta_{\mathcal{R}}$ of $\R^m\times \R^{m\times n}$. Similarly, $h(x,r,s):=|r|^q$ is not admissible, either. Note that this problem is completely independent of the choice of compactifications.
\end{remark}
In view of the issue pointed out in Remark~\ref{rem:catch}, we will not use Theorem~\ref{thm0} and its class of anisotropic DiPerna-Majda measures below. Instead, our next statement provides 
an alternative approach 
which in particular does allow integrands of the form $h(x,r,s):=|s|^p$.
\begin{theorem}\label{thm1}
Let  $1\le q\le +\infty$ and $1\le p<+\infty$. Let  $\{u_k\}_{k\in\N}$ be  bounded sequence in  $L^q(\O;\R^m)$ and $\{w_k\}$ a bounded sequence in $L^p(\O;\R^{m\times n})$. Then there is a (non-relabeled) subsequence $\{(u_k,w_k)\}$,  a DiPerna-Majda measure
$(\sigma,\hat\nu)\in\cdm$ and $\hat\mu\in\mathcal{Y}(\bar\O\times \b_\mathcal{R}\R^{m\times n};\b_\mathcal{U}\R^m)$, such that for every $f_0\in\mathcal{U}$, every $\psi_0\in\mathcal{R}$ and every  $g\in C(\bar\O)$ 
\begin{align}\label{generate}
\begin{aligned}
&\lim_{k\to\infty}\int_\O g(x) f_0(u_k(x))\psi(w_k(x))\,\md x\\
&\qquad=\int_{\bar\O}\int_{\b_\mathcal{R}\R^{m\times n}}\int_{\b_\mathcal{U}\R^m}g(x)f_0 (r)\psi_0(s)\hat\mu_{s,x}(\md r)\hat\nu_x(\md s)\sigma(\md x)\ ,
\end{aligned}
\end{align}
where $\psi(s):=\psi_0(s)(1+|s|^p)$. Moreover, measure $(\sigma,\hat{\nu})$ is generated by 
$\{ w_k\}$.
\end{theorem}
{\bf Proof.}
Due to separability of $\mathcal{U}$, $\mathcal{R}$ and of $C(\bar\O)$ there is a  (non-relabeled) subsequence of $\{(u_k,w_k)\}$ such that for all $[g\otimes f_0\otimes \psi_0] \in C(\bar\O)\times C(\beta_\mathcal{U}\R^m)\times C(\beta_\mathcal{R}\R^{m\times n})$  
and $\psi(s):=\psi_0(s)(1+|s|^p)$
\begin{align}\label{limit1}
\lim_{k\to\infty}\int_\O g(x) f_0(u_k(x))\psi(w_k(x))\,\md x=\la\Lambda, g\otimes f_0\otimes \psi_0\ra\ ,
\end{align}
for some $\Lambda\in\mathcal{M}(\bar\O\times \b_\mathcal{U}\R^m\times\b_\mathcal{R}\R^{m\times n})$.

We further define $\hat T_\Lambda:\mathcal{U}\times \mathcal{R}\to C(\bar\O)^*=\rca(\bar\O)$  by   $\la \hat T_\Lambda(f_0,\psi_0), g\ra:=\la\Lambda,g\otimes f_0\otimes \psi_0\ra$.
Let $\sigma\in\rca(\bar\O)$ be the weak* limit of $\{1+|w_k|^p)\}$. 
Then we see that due to \eqref{limit1}
\begin{align} \label{Tbound}
|\la \hat T_\Lambda(f_0,\psi_0), g\ra|= |\la\Lambda,g\otimes f_0\otimes \psi_0\ra|\le \|f_0\|_{C(\R^m)} \|\psi_0\|_{C(\R^{m\times n})}\int_{\bar\O} g(x)\,\sigma(\md x)\ .
\end{align}
This means that $\hat T_\Lambda(f_0,\psi_0)$ is absolutely continuous with respect to $\sigma$ and by the Radon-Nikod\'{y}m theorem 
there is $T_\Lambda:\mathcal{U}\times\mathcal{R}\to L^1(\bar\O;\sigma)$ such that for any Borel subset $\omega\subset\bar\O$  we get 
$\hat T_\Lambda(f_0,\psi_0)(\omega)=\int_\omega T_\Lambda(f_0,\psi_0)(x)\sigma(\md x)$. Consequently, 
the right-hand side of \eqref{limit1} can be written as 
$\int_{\bar\O} T_\Lambda(f_0,\psi_0)(x)g(x)\sigma(\md x)$. 

As $\mathcal{U}\times\mathcal{R}$ is separable, $\b_\mathcal{U}\R^m\times\b_\mathcal{R}\R^{m\times n}$ is metrizable and separable (with $\R^m\times\R^{m\times n}$ a dense subset) and $\sigma$ is a regular measure, the linear span of $C(\bar\O)\otimes C(\b_\mathcal{U}\R^m) \otimes C(\b_\mathcal{R}\R^{m\times n})$ is dense in  $L^1(\bar\O,\sigma;C(\b_\mathcal{U}\R^m\times\b_\mathcal{R}\R^{m\times n}))$ \cite[Thm.~1.5.25]{warga}.  Because of this and \eqref{Tbound}, 
$\Lambda$ can be continuously extended to a continuous linear functional on $L^1(\bar\O,\sigma;C(\b_\mathcal{U}\R^m\times\b_\mathcal{R}\R^{m\times n}))$. However, 
the dual of this space is isometrically isomorphic to $L^{\infty}_{\rm w}(\bar{\O},\s;\rca(\b_\mathcal{U}\R^m\times \b_{\cal R} \R^{m\times n}))$.
Arguing as in \cite[p.~133]{r} we get that there is a family $\lambda:=\{\lambda_x\}_{x\in\bar\O}$ of probability measures on $ \b_\mathcal{U}\R^m\times \b_{\cal R} \R^{m\times n}$ which is $\sigma$-weak* measurable,  for any $z\in C(\b_\mathcal{U}\R^m\times\b_{\cal R}\R^{m\times n})$, the mapping
$\bar\O\to\R:x\mapsto\int_{\b_\mathcal{U}\R^m\times\b_{\cal R}\R^{m\times n}} z(r,s)\lambda_x(\md r\d s)$ is $\s$-measurable in
the usual sense. Moreover, for $\sigma$-almost all $x\in\bar\O$ it holds that 
\begin{align}
T_\Lambda(f_0,\psi_0)(x)=\int_{\b_\mathcal{U}\R^m\times\b_{\cal R}\R^{m\times n}}f_0 (r)\psi_0(s)\lambda_x(\md r\md s)\ .
\end{align}
Altogether, we see that \eqref{limit1} can be rewritten as 
\begin{align}\label{limit2}
&\lim_{k\to\infty}\int_\O g(x) f_0(u_k(x))\psi(w_k(x))\,\md x=\int_{\bar\O}g(x)\int_{\b_\mathcal{U}\R^m\times\b_{\cal R}\R^{m\times n}}f_0 (r)\psi_0(s)\lambda_x(\md r\md s)\sigma(\md x)\ .
\end{align}
Applying the slicing-measure decomposition \cite[Thm.~1.5.1]{evans0} to each $\lambda_x$ we write 
$\lambda_x(\md r\md s)=\hat\mu_{s,x}(\md r)\hat\nu_{x}(\md s)$. As $\lambda_x$ is a probability measure we get that both  
$\hat\mu_{s,x}$ as well as $\hat\nu_x$ are probability measures on $\b_\mathcal{U}\R^m$ and $\b_\mathcal{R}\R^{m\times n}$, respectively.
Plugging  this decomposition into \eqref{limit2} and testing it with $f_0:=1$, we get 
\begin{align}\label{limit3}
\lim_{k\to\infty}\int_\O g(x) \psi(w_k(x))\,\md x=\int_{\bar\O}g(x)\int_{\b_{\cal R}\R^{m\times n}}\psi_0(s)\hat\nu_x(\md s)\sigma(\md x)\ .
\end{align}
This means that $(\sigma,\hat\nu)$ is the  DiPerna-Majda measure \emph{generated by $\{w_k\}$} \cite{diperna-majda}. 
\hfill$\Box$

In the situation of Theorem~\ref{thm1}, passing to a subsequence (not relabeled) if necessary, we may assume in addition that $\{(u_k,w_k)\}$ generates the (classical) Young measure $\xi_x$. Using the slicing-measure decomposition \cite[Thm.~1.5.1]{evans0} as before, we can always decompose
$\xi_x(\md(r,s))=\mu_{x,s}(\md r) \nu_x(\md s)$, so that
$$
\begin{aligned}
	\int_\O g(x) f_0(u_k) \psi_0(w_k)\,\md x\to  &\int_\O \int_{\R^m\times \R^{m\times n}} g(x)f_0(r)\psi_0(s)\,\xi_x(\md(r,s))\md x \\
	&=\int_\O \int_{\R^{m\times n}} \int_{\R^m} g(x)f_0(r)\psi_0(s)\,\mu_{x,s}(\md r)\nu_x(\md s)\md x,
\end{aligned}
$$
in particular for every $f_0\in\mathcal{U}$, every $\psi_0\in\mathcal{R}$ and every $g\in C(\bar\O)$.
The link between $(\mu,\nu)$ and $(\hat\mu,\hat\nu)$ is the following:
\begin{corollary}\label{cor:YMpair-DMpair}
In the situation of Theorem~\ref{thm1}, let $\xi_x(\md(r,s))=\mu_{x,s}(\md r) \nu_x(\md s)$ be the Young measure generated by $\{(u_k,w_k)\}$.
Then $\md x=\left(\int_{\R^{m\times n}}\frac{1}{1+|t|^p}\hat\nu_x(\md t)\right)\,\sigma(\md x)$, and for a.e.~$x\in \O$,
\begin{align}\label{YMpair-DMpair-1}
  \nu_x(\md s)=\left(\int_{\R^{m\times n}}\frac{1}{1+|t|^p}\hat\nu_x(\md t)\right)^{-1}\frac{\hat\nu_x(\md s)}{1+|s|^p}
\end{align}
(this is actually the well known connection between the DiPerna-Majda-measure and the associated Young measure) and 
\begin{align}\label{YMpair-DMpair-2}
	\mu_{x,s}=\hat\mu_{x,s}~~\text{for $\hat\nu_x$-a.e.~$s\in \R^{m\times n}$}
\end{align}
\end{corollary}

\begin{proof}%[of Corollary~\ref{cor:YMpair-DMpair}]
In the following, let $\psi_0\in C_0(\R^{m\times n})$, i.e., $\psi_0\in\mathcal{R}$ with the added property that $\psi_0(s)=0$ for every $s\in \beta_{\mathcal R}\R^{m\times n}\setminus \R^{m\times n}$.
Consequently, $\psi(s):=\psi_0(s)(1+|s|^p)$ satisfies $(1+|s|^p)^{-1}\psi(s)\to 0$ as $|s|\to \infty$ ($s\in \R^{m\times n}$) and 
$\frac{\psi(s)}{1+|s|^p}=0$ for $s\in \b_\mathcal{R}\R^{m\times n}\setminus \R^{m\times n}$.
In addition, let $g\in C(\bar\O)$ and $f_0\in \mathcal{U}$.
From \eqref{limit2}, also using the decomposition $\lambda_x(\md r\md s)=\hat\mu_{s,x}(\md r)\hat\nu_{x}(\md s)$,
we get that 
\begin{align}\label{corYMDM-1}
\lim_{k\to\infty}\int_\O g(x) f_0(u_k(x))\psi(w_k(x))\,\md x
=\int_{\bar\O} g(x)\int_{\R^{m\times n}}\int_{\b_\mathcal{U}\R^m} f_0 (r)\hat\mu_{s,x}(\md r)\frac{\psi(s)\hat\nu_x(\md s)}{1+|s|^p}\sigma(\md x)\ .
\end{align}
Moreover, since $f_0$ is bounded, $\{w_k\}$ is bounded in $L^p$ and $\psi$ has less than $p$-growth, the left hand side can be expressed using the Young measure $\xi_x(\md(r,s))=\mu_{x,s}(\md r) \nu_x(\md s)$ generated by $\{(u_k,w_k)\}$:
\begin{align}\label{corYMDM-2}
\lim_{k\to\infty}\int_\O g(x) f_0(u_k(x))\psi(w_k(x))\,\md x
=\int_{\O}g(x)\int_{\R^{m\times n}}\int_{\R^m} f_0 (r)\mu_{s,x}(\md r) \psi(s)\nu_x(\md s) \md x\ .
\end{align}
Since $(\sigma,\hat\nu)$ is a DiPerna-Majda measure (the one generated by $\{w_k\}$), we in particular know that
the density of the Lebesgue measure with respect to $\sigma$ is given by
$$
\frac{\md\mathcal{L}^n}{\md\sigma}(x)=\int_{\R^{m\times n}}\frac{\hat\nu_x(\md s)}{1+|s|^p}\ ,
$$
cf.~Proposition~\ref{characterization} (ii).
Hence, we can also write the outer integral on right hand side of \eqref{corYMDM-2} as an integral with respect to $\sigma$, and then compare it to the right hand side of \eqref{corYMDM-1}. 
Since $g$ is arbitrary, this implies that for $\sigma$-a.e.~$x\in\O$, 
\begin{align}
\begin{aligned}\label{corYMDM-3}
\Big(\int_{\R^{m\times n}}\int_{\R^m} f_0 (r)\mu_{s,x}(\md r) \psi(s)\nu_x(\md s)\Big)
\Big(\int_{\R^{m\times n}}\frac{\hat\nu_x(\md t)}{1+|t|^p}\Big)&\\
\qquad=
\int_{\R^{m\times n}}\int_{\b_\mathcal{U}\R^m} f_0 (r)\hat\mu_{s,x}(\md r)\frac{\psi(s)\hat\nu_x(\md s)}{1+|s|^p}&\ .
\end{aligned}
\end{align}
Here, also notice that it is enough to state \eqref{corYMDM-3} for a.e.~$x\in\O$, because $\mathcal{L}^n$ is absolutely continuous with respect to $\sigma$ and $\int_{\R^{m\times n}}\frac{\hat\nu_x(\md t)}{1+|t|^p}=0$ for $\sigma^s$-a.e.~$x\in \bar\O$. 

Using the probability measure given by the right hand side of \eqref{YMpair-DMpair-1}, i.e.,
\[
	\nu_x(\md s):=\Big(\int_{\R^{m\times n}}\frac{\hat\nu_x(\md t)}{1+|t|^p}\Big)^{-1}\frac{\hat\nu_x(\md s)}{1+|s|^p},
\]
we see that \eqref{corYMDM-3} is equivalent to
\begin{align}
\begin{aligned}\label{corYMDM-4}
\int_{\R^{m\times n}}\int_{\R^m} f_0 (r)\mu_{s,x}(\md r) \psi(s)\nu_x(\md s)
=
\int_{\R^{m\times n}}\int_{\b_\mathcal{U}\R^m} f_0 (r)\hat\mu_{s,x}(\md r)\psi(s)\tilde\nu_x(\md s)&\ .
\end{aligned}
\end{align}
Since \eqref{corYMDM-4} holds for all $\psi_0\in C_0(\R^{m\times n})$ (and therefore all $\psi$ with less than $p$-growth, in particular all bounded $\psi$) and $\mu_{s,x}$ and $\hat\mu_{s,x}$ are probability measures, choosing $f_0\equiv 1\in \mathcal{U}$ 
in \eqref{corYMDM-4} yields that $\nu_x=\tilde\nu_x$, i.e., \eqref{YMpair-DMpair-1}. 
Finally, replacing $\tilde\nu_x$ by $\nu_x$ in \eqref{corYMDM-4}, and using that the latter holds in particular for all bounded $\psi\in C(\R^{m\times n})$
and all $f_0\in C_0(\R^m)\subset \mathcal{U}$, we infer \eqref{YMpair-DMpair-2}.
\end{proof}

\begin{remark}\label{rem:DMpair-Dirac}
In the situation of Corollary~\ref{cor:YMpair-DMpair}, suppose in addition that $u_k\to u$ in $L^q$ for some $q\geq 1$ (for instance by compact embedding, if $\{u_k\}$ is bounded in $W^{1,p}$).
We recall that in this case, for the Young measure $\xi_x(\md (r,s))=\mu_{x,s}(\md r)\nu_x (\md s)$ generated by $\{(u_k,w_k\}$ we have $\mu_{x,s}=\delta_{u(x)}$ for a.e.~$x\in\O$ (in particular independent of $s$, cf.~\cite[Proposition 6.13]{pedregal}, e.g.).
Consequently, \eqref{YMpair-DMpair-2} implies that
\begin{align}\label{YMidentity2}
	\hat\mu_{x,s}=\delta_{u(x)}~~\text{for a.e.~$x\in \O$ and $\hat\nu_x$-a.e.~$s\in \R^{m\times n}$}
\end{align}
\end{remark}

\begin{remark}\label{uniform}
 It is left to the interested reader to show that if $u_k\to u$ in $C(\bar\O;\R^m)$ for $k\to\infty$ then $\hat\mu_{s,x}=\delta_{u(x)}$ for $\sigma$-a.e.~$x\in\bar\O$. Also, $\hat\mu_{s,x}$ is then
supported only on $\R^m$, so it is independent of the choice of the compactification  $\beta_{\mathcal{U}}\R^m$.
\end{remark}
The next statement is similar to Theorem \ref{thm1}, but now we consider the limits of the sequence  $\int_\O f_0(u_k)\psi_0(w_k)(1+|u_k|^q)\,\md x$ where $f_0\in \mathcal{U}, \psi_0\in\mathcal{R}$.
In particular, the integrand $|u_k|^q$ will thus be admissible.
Its proof can easily be deduced by adapting the proof of Theorem~\ref{thm1}, essentially interchanging the role of the two sequences.
\begin{theorem}\label{thm1b}
Let $1\le q< +\infty$ and $1\le p\le +\infty$. Let  $\{u_k\}_{k\in\N}$ be  bounded sequence in  $L^q(\O;\R^m)$ and $\{w_k\}$ a bounded sequence in $L^p(\O;\R^{m\times n})$. Then there is a (non-relabeled) subsequence $\{(u_k,w_k)\}$,
a positive measure $\sigma^*\in \mathcal{M}(\bar\O)$ and parametrized probability measures 
$\hat\nu^*\in \mathcal{Y}(\bar\O;\b_\mathcal{R}\R^{m\times n})$ (defined $\sigma^*$-a.e.) and $\hat\mu^*\in \mathcal{Y}(\bar\O\times\b_\mathcal{R}\R^{m\times n};\b_\mathcal{U}\R^{m})$ (defined $\sigma^*\otimes \hat\nu^*_x$-a.e.) such that for every $f_0\in\mathcal{U}$, every $\psi_0\in\mathcal{R}$ and every $g\in C(\bar\O)$ 
\begin{align}\label{generatestar}
\begin{aligned}
&\lim_{k\to\infty}\int_\O g(x) f(u_k(x))\psi_0(w_k(x))\,\md x\\
&\qquad=\int_{\bar\O}\int_{\b_\mathcal{R}\R^{m\times n}}\int_{\b_\mathcal{U}\R^m}g(x)f_0 (r)\psi_0(s)\hat\mu_{s,x}^*(\md r)\hat\nu_x^*(\md s)\sigma^*(\md x)\ ,
\end{aligned}
\end{align}
where $f(r):=f_0(r)(1+|r|^q)$. Moreover, 
$(\sigma^*,\overline{\hat\mu^*}_{x})\in {\cal DM}^q_{\cal U}(\bar\O;\R^{m})$ is the 
the DiPerna-Majda measure generated by
$\{ u_k\}$, where $\overline{\hat\mu^*}_x$ is given as follows:
\begin{align}\label{barmustar}
\int_{\b_\mathcal{U}\R^m} f_0(r)\overline{\hat\mu^*}_x(\md r)=
\int_{\b_\mathcal{R}\R^{m\times n}} \int_{\b_\mathcal{U}\R^m} f_0(r) \hat\mu_{s,x}^*(\md r) \hat\nu_x^*(\md s)
\end{align} 
for all $f_0\in \mathcal{U}$ and $\sigma^*$-a.e.~$x\in\bar \O$.
\end{theorem}
Analogously to Corollary~\ref{cor:YMpair-DMpair}, we have
\begin{corollary}\label{cor:YMpair-DMpair-b}
In the situation of Theorem~\ref{thm1b}, let $\xi_x(\md (r,s))=\mu_{x,s}(\md r)\nu_x (\md s)$ be the Young measure generated by $\{(u_k,w_k)\}$.
Then $\md x=\left(\int_{\beta_\mathcal{R}\R^{m\times n}}\int_{\R^{m}}\frac{1}{1+|z|^q}\hat\mu_{x,t}^*(\md z)\hat\nu_x^*(\md t)\right)\,\sigma^*(\md x)$, and
for a.e.~$x\in \O$,
\begin{align}\label{YMpair-DMpair-1b}
  \nu_x(\md s)=
  %\beta_\mathcal{R}
  \left(\int_{\beta_\mathcal{R}\R^{m\times n}}\int_{\R^{m}}\frac{1}{1+|z|^q}\hat\mu_{x,t}^*(\md z)\hat\nu_x^*(\md t)\right)^{-1}
  \left(\int_{\R^{m}}\frac{1}{1+|z|^q}\hat\mu_{x,s}^*(\md z)\right)
  \hat\nu_x^*(\md s),
\end{align}
\begin{align}\label{YMpair-DMpair-2b}
	\mu_{x,s}(\md r)=
	\left(\int_{\R^{m}}\frac{1}{1+|z|^q}\hat\mu_{x,s}^*(\md z)\right)^{-1}\frac{\hat\mu_{x,s}(\md r)}{1+|r|^q}~~\text{for $\hat\nu_x$-a.e.~$s\in \beta_{\mathcal{R}}\R^{m\times n}$.}
\end{align}
\end{corollary}

Analogous to the case of Young measures or DiPerna-Majda-measures, we say that $(\sigma,\hat\nu,\hat\mu)$ [or $(\sigma^*,\hat\nu^*,\hat\mu^*)$, respectively] \emph{is generated by $\{(u_k,w_k)\}$}
 whenever \eqref{generate} [(\eqref{generatestar}] holds for all $(g,f_0,\psi_0)\in C(\bar\O)\times \mathcal{U}\times \mathcal{R}$.

Theorem~\ref{thm1} and Theorem~\ref{thm1b} can be combined, leading
to the following statement. It provides a representation for limits of rather general nonlinear functionals along a given sequence. The suitable class of integrands is 
\begin{align}\label{defHpq}
\HH^{q,p}(\O,\mathcal{U},\mathcal{R})=\left\{~ 
h
~\left|
~\begin{array}{ll}
h(x,r,s)=h_0^{(1)}(x,r,s)(1+|s|^p)+h_0^{(2)}(x,r,s)(1+|r|^q)\\
h_0^{(1)},h_0^{(2)}\in C(\bar\O\times \b_\mathcal{U}\R^{m}\times \b_\mathcal{R}\R^{m\times n})
\end{array}\right.\right\} .
\end{align}
\begin{theorem}[representation theorem]\label{thm-rep}
Let  $1\le q\le +\infty$ and $1\le p<+\infty$. Let  $\{u_k\}_{k\in\N}$ be  bounded sequence in  $L^q(\O;\R^m)$ and $\{w_k\}$ a bounded sequence in $L^p(\O;\R^{m\times n})$. Then there is a (non-relabeled) subsequence $\{(u_k,w_k)\}$ 
generating the measures $(\sigma,\hat\nu,\hat\mu)$ and $(\sigma^*,\hat\nu^*,\hat\mu^*)$ (in the sense of \eqref{generate} and \eqref{generatestar}, respectively), and in addition, 
for every $h_0^{(1)},h_0^{(2)}\in C(\bar\O\times \b_\mathcal{U}\R^{m}\times \b_\mathcal{R}\R^{m\times n})$, 
\begin{align}\label{generateall}
\begin{aligned}
&\lim_{k\to\infty}\int_\O \big(h_0^{(1)}(x,u_k,w_k)(1+ |w_k|^p)+h_0^{(2)}(x,u_k,w_k)(1+ |u_k|^q)  )\big)\,\md x\\
& \qquad= \begin{aligned}[t]
& \int_{\bar\O}\int_{\b_\mathcal{R}\R^{m\times n}}\int_{\b_\mathcal{U}\R^m}h_0^{(1)} (x,r,s)\hat\mu_{s,x}(\md r)\hat\nu_x(\md s)\sigma(\md x)\  \\
&  +\int_{\bar\O}\int_{\b_\mathcal{R}\R^{m\times n}}\int_{\b_\mathcal{U}\R^m}h_0^{(2)} (x,r,s)\hat\mu^*_{s,x}(\md r)\hat\nu^*_x(\md s)\sigma^*(\md x)\ .
\end{aligned}
\end{aligned}
\end{align}
%the following measures:
%\begin{itemize}
%\item $(\sigma,\hat\nu)\in\cdm$ - 
% a DiPerna-Majda measure generated by $\{ w_k\}$ , $\hat\mu\in\mathcal{Y}(\bar\O\times \b_\mathcal{R}\R^{m\times n};\b_\mathcal{U}\R^m)$; 
%\item $(\sigma^*,\overline{\hat\mu^*}_{x})\in {\cal DM}^q_{\cal U}(\bar\O;\R^{m})$ - a DiPerna-Majda measure generated by by $\{ u_k\}$
%\item 
% parametrized probability measures 
%$\hat\nu^*\in \mathcal{Y}(\bar\O;\b_\mathcal{R}\R^{m\times n})$ (defined $\sigma^*$-a.e.), $\hat\mu^*\in \mathcal{Y}(\bar\O\times\b_\mathcal{R}\R^{m\times n};\b_\mathcal{U}\R^{m})$ (defined $\sigma^*\otimes \hat\nu^*_x$-a.e.), 
% linked by
%$$\int_{\b_\mathcal{U}\R^m} f_0(r)\overline{\hat\mu^*}_{x}(\md r):=
%\int_{\b_\mathcal{U}\R^m} \int_{\b_\mathcal{R}\R^{m\times n}} f_0(r) \hat\mu_{s,x}^*(\md r) \hat\nu_x^*(\md s)$$ for all $f_0\in \mathcal{R}$ and $\sigma^*$-a.e.~$x\in\bar{\O}$
%\end{itemize}
\end{theorem}
\begin{remark}
As a special case, we recover a representation of the limit for functionals with integrands in $\YY^{q,p}(\O;\mathcal{U};\mathcal{R})$ as in Theorem~\ref{thm0}, since
$$
\tilde{h}_0(x,r,s)(1+|r|^q+|s|^p)=h_0(x,r,s)(1+|r|^q)+h_0(x,r,s)(1+|s|^p),
$$
where
$$
 h_0(x,r,s):=\frac{1+|r|^q+|s|^p}{2+|r|^q+|s|^p}\tilde{h}_0(x,r,s)
$$
The quotient which appears here does not matter, because
$(r,s)\mapsto \frac{1+|r|^q+|s|^p}{2+|r|^q+|s|^p}$ converges to the constant $1$ as $|(r,s)|\to \infty$, and therefore it is an element of $\overline{\mathcal{U}\otimes \mathcal{R}}=C(\b_\mathcal{U}\R^{m}\times \b_\mathcal{R}\R^{m\times n})$.
\end{remark}
\begin{remark}
Notice that $(\sigma,\hat\nu,\hat\mu)$ and $(\sigma^*,\hat\nu^*,\hat\mu^*)$ are not independent, because
they share the same underlying Young measure $\xi_x(\md (r,s))=\mu_{x,s}(\md r)\nu_x (\md s)$, see Corollary~\ref{cor:YMpair-DMpair} and Corollary~\ref{cor:YMpair-DMpair-b}.
Using that, we get yet another representation: For $h\in \HH^{q,p}(\O,\mathcal{U},\mathcal{R})$ (cf.~\eqref{defHpq}),
\begin{align}\label{generateall2}
\begin{aligned}
&\lim_{k\to\infty}\int_\O h(x,u_k,w_k)\,\md x\\
&=\begin{aligned}[t]
&\int_{\bar\O}\int_{\b_\mathcal{R}\R^{m\times n}\setminus \R^{m\times n} }\int_{\b_\mathcal{U}\R^m}h_0^{(1)} (x,r,s)\hat\mu_{s,x}(\md r)\hat\nu_x(\md s)\sigma(\md x)\  \\
&+\int_{\bar\O}\int_{\b_\mathcal{R}\R^{m\times n}}\int_{\b_\mathcal{U}\R^m\setminus \R^m}h_0^{(2)} (x,r,s)\hat\mu^*_{s,x}(\md r)\hat\nu^*_x(\md s)\sigma^*(\md x)\ \\
&+\int_{\O}\int_{\R^{m\times n}}\int_{\R^m} h(x,r,s)
 \mu_{s,x}(\md r)\nu_x(\md s)\md x.
\end{aligned}
\end{aligned}
\end{align}
\end{remark}
\begin{remark}\label{rem:simplify}
If either $\{|u_k|^q\}$ or $\{|w_k|^q\}$ is equi-integrable, then \eqref{generateall2} can be further simplified. 
For instance, if $\{u_k\}$ is bounded in $L^{\tilde{q}}$ for some $\tilde{q}>q$), 
then $\{|u_k|^q\}$ is equi-integrable, and it that case, it is known (e.g., see \cite[Lemma 3.2.14]{r}) that for the associated DiPerna-Majda measure $(\sigma^*,\overline{\hat\mu^*}_x)$, we have that
$\sigma^*$ is absolutely continuous with respect to $\mathcal{L}^n$ and 
$\overline{\hat\mu^*}_x(\b_\mathcal{U}\R^m\setminus \R^m)=0$ for a.e.~$x\in\O$. Due to 
\eqref{barmustar}, the latter implies that
$\hat\mu_{x,s}^*(\b_\mathcal{U}\R^m\setminus \R^m)=0$ for a.e.~$x\in\O$ and $\hat\nu_x^*$-a.e.~$s\in \b_\mathcal{R}\R^{m\times n}$. Accordingly, 
for $h\in \HH^{q,p}(\O,\mathcal{U},\mathcal{R})$ (cf.~\eqref{defHpq}), 
\begin{align}\label{generateall3}
\begin{aligned}
\lim_{k\to\infty}\int_\O h(x,u_k,w_k)\,\md x
&=\begin{aligned}[t]
&\int_{\bar\O}\int_{\b_\mathcal{R}\R^{m\times n}\setminus \R^{m\times n} }\int_{\b_\mathcal{U}\R^m}h_0^{(1)} (x,r,s)\hat\mu_{s,x}(\md r)\hat\nu_x(\md s)\sigma(\md x)\  \\
&+\int_{\O}\int_{\R^{m\times n}}\int_{\R^m} h(x,r,s)
\mu_{s,x}(\md r)\nu_x(\md s)\md x.
\end{aligned}
\end{aligned}
\end{align}
\end{remark}

\subsection{Analysis for couples $\{ (u_k,\nabla u_k)\}$}
For the rest of the article, we are mainly interested in sequences of the form $(u_k,w_k)=(u_k,\nabla u_k)$, with a bounded sequence $\{u_k\}\subset W^{1,p}(\O;\R^m)$, $1\leq p<\infty$, and integrands $h\in \HH^{q,p}(\O,\mathcal{U},\mathcal{R})$ (cf.~\eqref{defHpq}) for some $q<p^*$. Here, $p^*$ is the exponent of the Sobolev embedding, i.e.,
$$
p^*:=\begin{cases}
pn/(n-p) &\text{ if $1\le p<n$},\\
+\infty & \text{otherwise.}
\end{cases}
$$
In particular, such integrands satisfy
\begin{align}\label{h-growth}
|h(x,r,s)|\le C(1+|r|^{q}+|s|^p)\quad\text{for all $x\in \bar\O$, $r\in \R^m$, $s\in \R^{m\times n}$},
\end{align}
with a constant $C\geq 0$.

Since we assume that $q<p^*$, we can represent limits using \eqref{generateall3}, with the added observation that the Young measure generated by 
$\{u_k\}$ is given by $\delta_{u(x)}$ (whence $\mu_{x,s}=\delta_{u(x)}$ for all $s$), where $u$ denotes the weak limit of $\{u_k\}$ in $W^{1,p}(\O;\R^m)$. This gives the following result.
\begin{theorem}\label{thm:repgrad}
Let $(u_k,w_k):=(u_k,\nabla u_k)$, with a bounded sequence $\{u_k\}\subset W^{1,p}(\O;\R^m)$, $1\leq p<\infty$, such that $u_k\rightharpoonup u$
in $W^{1,p}(\O;\R^m)$, $\{(\nabla u_k)\}$ generates the 
(classical) Young measure $\nu_x$ in the sense of \eqref{young} and
$\{(u_k,\nabla u_k)\}$ generates the
measure
$(\sigma,\hat\nu,\hat\mu)$ in the sense of \eqref{generate}. Then for every $h\in \HH^{q,p}(\O,\mathcal{U},\mathcal{R})$ (cf.~\eqref{defHpq}), 
\begin{align}\label{limit21}
&\lim_{k\to\infty}  \int_{\O}h(x,u_k(x),\nabla u_k(x))\,\md x\nonumber\\
&=\int_{\O}\int_{\R^{m\times n}} h(x,u(x),s)\nu_x(\md s)\,\md x\nonumber\\
&\quad+\int_{\bar\O}\int_{\b_\mathcal{R}\R^{m\times n}\setminus\R^{m\times n}}\int_{\b_\mathcal{U}\R^m}h_0^{(1)}(x,r,s)\hat\mu_{s,x}(\md r)\hat\nu_x(\md s)\sigma(\md x)\ .
\end{align}
\end{theorem}
\begin{remark}\label{rem:qc}
If $h(x,u(x),\cdot)$ is quasiconvex, we can further calculate in \eqref{limit21} as follows:
\begin{align}\label{useqc0}
&\int_{\O}\int_{\R^{m\times n}} h(x,u(x),s)\nu_x(\md s)\,\md x %\nonumber\\
%&\qquad+\int_{\bar\O}\int_{\b_\mathcal{R}\R^{m\times n}\setminus\R^{m\times n}}\int_{\b_\mathcal{U}\R^m}\frac{h(x,r,s)}{1+|r|^q+|s|^p}\hat\mu_{s,x}(\md r)\hat\nu_x(\md s)\sigma(\md x)\nonumber\\
%&
\ge \int_{\O} h(x,u(x),\nabla u(x))\,\md x.
%\nonumber\\
%&\qquad +\int_{\bar\O}\int_{\b_\mathcal{R}\R^{m\times n}\setminus\R^{m\times n}}\int_{\b_\mathcal{U}\R^m}\frac{h(x,r,s)}{1+|r|^q+|s|^p}\hat\mu_{s,x}(\md r)\hat\nu_x(\md s)\sigma(\md x)
\end{align}
\end{remark}
\begin{remark}\label{rem:uniform2}
If $p>n$, $W^{1,p}(\O;\R^m)$ is compactly embedded in $C(\bar\O;\R^m)$, and therefore $u_k\to u$ uniformly on $\bar\O$.
In view of Remark~\ref{uniform}, we then have that $\hat\mu_{s,x}=\delta_{u(x)}$ for $\sigma$-a.e.~$x\in\bar\O$, for $\hat\nu_x$-a.e.~$s\in \beta_\mathcal{R} \R^{m\times n}$. Hence, 
\[
  \int_{\b_\mathcal{U}\R^m}h_0^{(1)}(x,r,s)\hat\mu_{s,x}(\md r)=h_0^{(1)}(x,u(x),s)
\]
in the right hand side of \eqref{limit21}.
\end{remark}

\subsection{Examples}

Below, we give a couple of examples of sequences and measures from Theorem~\ref{thm1} generated by them.

\begin{example}\label{ex-first}
Let $u_k\in W^{1,1}(0,2)$ be such that
$$
u_k(x):=\begin{cases}
0 &\text{ if $0\le x\le 1-1/k$},\\
kx-k+1 &\text{ if $1-1/k\le x\le 1$},\\
-2kx+2k+1 &\text{ if $1\le x\le 1+1/k$},\\
-1 & \text{ if $1+1/k\le x\le 2$.}
\end{cases}
$$
Let $w_k:=u_k'$, i.e.,
$$
w_k(x):=\begin{cases}
0 &\text{ if $0\le x\le 1-1/k$},\\
k &\text{ if $1-1/k\le x\le 1$},\\
-2k &\text{ if $1\le x\le 1+1/k$},\\
0 & \text{ if $1+1/k\le x\le 2$.}
\end{cases}
$$

\begin{figure}[b]
\sidecaption
% For example, with the graphicx style use
\includegraphics[scale=.65]{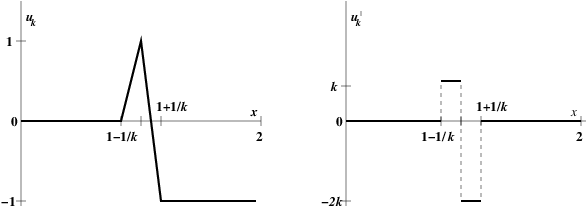}
%
% If no graphics program available, insert a blank space i.e. use
%\picplace{5cm}{2cm} % Give the correct figure height and width in cm
%
\caption{Sequence $\{u_k,u_k'\}_{k\in\N}$ from Example~\ref{ex-first}.}
\label{fig:1}
        % Give a unique label
\end{figure}
Let $f_0\in C(\R)$ be bounded with its primitive denoted by $F$, i.e., $F'=f_0$, $g\in C(\bar\O)$,  and let $\psi=\psi_0(1+|\cdot|)$ where
$\psi_0\in \mathcal{R}$ corresponding to the two-point (or sphere) compactification $\beta_{\mathcal{R}}\R=\R\cup \{\pm \infty\}$, i.e., $\psi_0\in C(\R)$  is such that $\lim_{s\to\pm\infty}\psi_0(s)=:\psi_0(\pm\infty)\in\R$.
Then
\begin{align*}
&\lim_{k\to\infty}\int_0^2f_0(u_k(x))\psi(w_k(x))g(x)\,\md x\\
&=\begin{aligned}[t]
&\lim_{k\to\infty}\Big(\int_0^{1-1/k}f_0(0)\psi_0(0)g(x)\,\md x+\int_{1+1/k}^2f_0(-1)\psi_0(0)g(x)\,\md x\Big)\\
&+\begin{aligned}[t]\lim_{k\to\infty}\Big(&
\int_{1-1/k}^1f_0(kx-k+1)\psi_0(k)(1+k) g(x)\,\md x\\
&+\int_{1}^{1+1/k}f_0(-2kx+2k+1)\psi_0(-2k)(1+2k)g(x)\,\md x\Big)
\end{aligned}
\end{aligned}\nonumber\\
&=\begin{aligned}[t]
&\psi_0(0)(f_0(0)\int_0^1g(x)\,\md x+f_0(-1)\int_1^2g(x)\,\md x)\\
&+\begin{aligned}[t]\lim_{k\to\infty}\Big(&\int_{1-1/k}^1[F(kx-k+1)]'\psi_0(k)\frac{(1+k)}{k}g(x)\,\md x\\
&+\int_{1}^{1+1/k}[F(-2kx+2k+1)]'\psi_0(-2k)\frac{(1+2k)}{-2k}g(x)\,\md x\Big)
\end{aligned}
\end{aligned}\nonumber\\
& =\begin{aligned}[t]
&f_0(0)\psi_0(0)\int_0^1g(x)\,\md x+f_0(-1)\psi_0(0)\int_1^2g(x)\,\md x\nonumber\\
& +g(1)(F(1)-F(0))\psi_0(+\infty)+g(1)(F(1)-F(-1))\psi_0(-\infty)
\end{aligned}\nonumber\\
&=\int_0^2\int_{\beta_\mathcal{U}\R}\int_{\beta_\mathcal{R}\R}g(x)f_0(r)\psi_0(s)\hat{\mu}_{s,x}(\md r)\hat\nu_x(\md s)\sigma(\md x)\ ,
\end{align*}
where
$\sigma=\mathcal{L}^1+3\delta_1$,
$$
\hat\nu_x=\begin{cases}
\delta_0 &\text{ if $x\in[0,1)\cup(1;2]$},\\
\frac13\delta_\infty+\frac23\delta_{-\infty} & \text{ if $x=1$},
\end{cases}
$$
and
$$
\hat\mu_{s,x}=
\begin{cases}
\delta_0 & \text{ if $0\le x<1$},\\
\delta_{-1} &\text{ if $1<x\le 2$},\\
\mathcal{L}^1\llcorner_{(0,1)} & \text{ if $s=+\infty$ and $x=1$},\\
 \frac12\mathcal{L}^1\llcorner_{(-1,1)} & \text{ if $s=-\infty$ and $x=1$.}
\end{cases}
$$
\end{example}

Changing the previous sequence slightly we get the same measure $(\sigma,\hat\nu)$, the same limit of $\{u_k\}$  but a different measure $\hat\mu$.

\begin{example}
Let $u_k\in W^{1,1}(0,2)$ be such that
$$
u_k(x):=\begin{cases}
0 &\text{ if $0\le x\le 1-2/k$},\\
-kx+k-2 &\text{ if $1-2/k\le x\le 1-1/k$},\\
kx-k &\text{ if $1-1/k\le x\le 1$},\\
-kx+k & \text{ if $1\le x\le 1+1/k$},\\
-1 &\text{ if $1+1/k\le x\le 2$.}
\end{cases}
$$
Let $w_k:=u_k'$, i.e.,
$$
w_k(x):=\begin{cases}
0 &\text{ if $0\le x\le 1-2/k$},\\
-k &\text{ if $1-2/k\le x\le 1-1/k$},\\
k &\text{ if $1-1/k\le x\le 1$},\\
-k & \text{ if $1\le x\le 1+1/k$},\\
 0 &\text{ if $1+1/k\le x\le 2$.}
\end{cases}
$$

Then a computation analogous to the one above   shows that

$\sigma=\mathcal{L}^1+3\delta_1$,
$$
\hat\nu_x=\begin{cases}
\delta_0 &\text{ if $x\in[0,1)\cup(1;2]$},\\
\frac13\delta_\infty+\frac23\delta_{-\infty} & \text{ if $x=1$},
\end{cases}
$$
and
$$
\hat\mu_{s,x}=
\begin{cases}
\delta_0 & \text{ if $0\le x<1$},\\
\delta_{-1} &\text{ if $1<x\le 2$},\\
\mathcal{L}^1\llcorner_{(-1,0)} & \text{ if $s=-\infty$ and  $x=1$,  }\\
\mathcal{L}^1\llcorner_{(-1,0)} & \text{ if  $s=+\infty$ and  $x=1$,  }
\end{cases}
$$
\end{example}

These two examples show that $\hat\mu$ captures behavior of $\{u_k\}$ and cannot be read off either from  $(\sigma,\hat\nu)$ and/or from $u$.

\begin{example}
In the next example,  we just   set $u_k:=u$, where $u(x):=0$ if $x\in[0,1]$ and $u(x)=-1$ if $x\in(1;2]$,  and $\{w_k\}_{k\in\N}$ for all $k\in\N$ as before. This gives us
\begin{align*}
&\lim_{k\to\infty}\int_0^2f_0(u(x))\psi(w_k(x))g(x)\,\md x\\
&= \begin{aligned}[t]
&\lim_{k\to\infty}\Big(\int_0^{1-1/k}f_0(0)\psi_0(0)g(x)\,\md x+\int_{1+1/k}^2f_0(-1)\psi_0(0)g(x)\,\md x\Big)\\
&+\lim_{k\to\infty}\Big(\int_{1-1/k}^1f_0(0)\psi_0(k)(1+k)g(x)\,\md x+\int_{1}^{1+1/k}f_0(-1)\psi_0(-2k)(1+2k)g(x)\,\md x\Big)
\end{aligned}\nonumber\\
&= \begin{aligned}[t]
&f_0(0)\psi_0(0)\int_0^1g(x)\,\md x+f_0(-1)\psi_0(0)\int_1^2g(x)\,\md x\\
&+\lim_{k\to\infty}\Big(k\int_{1-1/k}^1 f_0(0)\psi_0(k)\frac{1+k}{k}g(x)\,\md x
+k\int_{1}^{1+1/k}f_0(-1)\psi_0(-2k)\frac{1+2k}{k}g(x)\,\md x\Big)
\end{aligned}\nonumber\\
&= \begin{aligned}[t]
&f_0(0)\psi_0(0)\int_0^1g(x)\,\md x+f_0(-1)\psi_0(0)\int_1^2g(x)\,\md x\\
&+g(1)f_0(0)\psi_0(+\infty)+2g(1)f_0(-1)\psi_0(-\infty))
\end{aligned}\nonumber\\
&=\int_0^2\int_{\beta_\mathcal{R}\R}\int_{\beta_\mathcal{U}}g(x)f_0(r)\psi_0(s)\nu_{s,x}(\md r)\hat\nu_x(\md s)\sigma(\md x)\ ,
\end{align*}
where
 $\sigma=\mathcal{L}^1+3\delta_1$,
$$
\hat\nu_x=\begin{cases}
\delta_0 &\text{ if $x\in[0,1)\cup(1;2]$},\\
\frac13\delta_\infty+\frac23\delta_{-\infty} & \text{ if $x=1$},
\end{cases}
$$
and
$$
\hat\mu_{s,x}=
\begin{cases}
\delta_0 & \text{ if $0\le x< 1$},\\
\delta_0 &\text{ if $x=1$ and $s=+\infty$},\\
\delta_{-1}& \text{ if $x=1$ and $s=-\infty$},\\
\delta_{-1} &\text{ if $1<x\le 2$}.\\
\end{cases}
$$

\end{example}

In the  example below,  we calculate the measure $\hat\mu$ of the strongly converging sequence.

\begin{example}\label{ex-simple}
Let $p=1$, consider the one-point compactification $\beta_{\mathcal{R}}\R=\R\cup \{\infty\}$ of $\R$, and let
 $\{u_k\}_{k\in\N}\subset W^{1,1}(0,2)$, $u_k\rightharpoonup u$, be a sequence of nondecreasing functions such that
 $u_k(0)=0$ and $u_k(2)=1$ for all $k\in\N$. In addition, suppose that
$\{u'_k\}_{k\in\N}\subset L^1(0,2)$ converges to zero in measure and it concentrates  at $x=1$,
i.e., $\{u'_k\}$ generates $(\sigma,\hat\nu)\in\cdm$ given by
$$
\sigma=\mathcal{L}^1+\delta_1,\quad \hat\nu_x=\begin{cases}
\delta_0 &\text{ if $x\in[0, 1)\cup(1,2]$},\\
  \delta_\infty & \text{ if $x=1$.}
\end{cases}
$$

Moreover, let $\alpha\ge 0$, let $f_0(r)\in C_0(\R)$ be such that
$$f_0(r)=
\begin{cases}
r^\alpha & \text{ if  $0\le r\le 1$}\\
1 & \text{ for $r\ge 1$ }%\\
%0 & \text{ for  $r\le 0$ }
\end{cases}
$$
and let  $\psi(s):=|s|$.
As $u_k$ is nondecreasing it must always satisfy $u_k\in [0,1]$, so that $f_0(u_k)=u_k^\alpha$, and $u_k'\ge 0$. Consequently, in view of Theorem~\ref{thm1}
\begin{align*}\lim_{k\to\infty}\int_0^2 f_0(u_k(x))\psi(u'_k(x))\,\md x&=\int_0^2\int_{\beta_{\mathcal{R}}\R}\int_{\beta_{\mathcal{U}}\R} r^\alpha\hat\mu_{s,x}(\md r)\frac{s}{1+|s|}\hat\nu_x(\md s)\sigma(\md x)\\
& = \int_{\beta_{\mathcal{U}}\R}r^\alpha\hat\mu_{\infty,1}(\md r)\ .\end{align*}
On the other hand,
\begin{align*}
\lim_{k\to\infty}\frac1{\alpha+1}(u_k^{\alpha+1}(2)-u_k^{\alpha+1}(0))&=\lim_{k\to\infty}\int_0^2 \frac1{\alpha+1}(u_k^{\alpha+1}(x))'\,\md x=\lim_{k\to\infty}\int_0^2 u^\alpha_k(x)u'_k(x)\,\md x\\
&= \int_{\beta_{\mathcal{U}}\R}r^\alpha\hat\mu_{\infty,1}(\md r)=\lim_{k\to\infty}\int_{u_k(0)}^{u_k(2)} r^\alpha\,\md r= \int_{0}^{1} 
r^\alpha \,\md r \ .
\end{align*}
Since $\alpha\geq 0$ is arbitrary and the polynomials are dense in the continuous functions on all compact subsets of $\R$, we infer that
$$\hat\mu_{s,x}=\begin{cases}
\delta_{u(x)} &\text{ if $x\in [0,1)\cup(1,2]$,}\\
\mathcal{L}^1\llcorner_{(0,1)} &\text{ if $x=1$ and $s=\infty$.}
\end{cases}
$$
This means that only values of limits of $u_k$ at $x=0$ and $x=2$ influence $\hat\mu_{\infty,1}$, i.e., the measure at the point where $\sigma$ concentrates.
\end{example}

\section{Applications to weak lower semicontinuity in Sobolev spaces}

We here focus on weak lower semikcontinuity of "signed" integral functionals in $W^{1,p}$, i.e., functional whose integrand may have a negative part which has $p$-growth in the gradient variable. 
The case of non-negative integrands (or weaker growth in the negative direction) is well-known, see e.g.~\cite{af}. 

Throughout this section, let $\mathcal{U}$ and $\mathcal{R}$ denote rings of bounded continuous functions corresponding to suitable metrizable compactifications $\beta_{\mathcal{U}}\R^m$ and $ \beta_{\mathcal{R}}\R^{m\times n}$ of $\R^m$ and $\R^{m\times n}$, respectively, as before.
The choice of these rings can be adapted to the particular integrand $h$ at hand in the results presented below. Compactifications by the sphere are sufficiently rich for most practical purposes.

If $p>n$, we can exploit the embedding of $W^{1,p}(\O;\R^m)$ into continuous functions on $\bar\O$. Still, even for quasiconvex integrands concentration effects near the boundary of the domain can prevent lower semicontinuity. However, as it turns out this is the only remaining obstacle. Unlike in the
related result of Ball and Zhang \cite{ball-zhang} where small measurable (but otherwise pretty unknown) sets are removed from the domain, for us it is enough to "peel" away a layer near $\partial \O$:
\begin{lemma}(Peeling lemma for $p>n$)
Let $\O\subset\R^n$ be a bounded domain with a boundary of class $C^1$, let $\infty>p>n$ and let 
$h\in \HH^{q,p}(\O,\mathcal{U},\mathcal{R})$ (cf.~\eqref{defHpq}). Moreover, assume that $h(x,r,\cdot)$ is quasiconvex for a.e.~$x\in\O$ (and therefore all $x\in\bar\O$, by continuity) and every $r\in \r^m$, and
let $\{u_k\}\subset W^{1,p}(\O;\R^m)$ be a bounded sequence with $u_k\rightharpoonup u$ in $W^{1,p}(\O;\R^m)$.
Then there exists an increasing sequence of open set $\O_j$ (possibly depending on the subsequence of $\{u_k\}$) 
with boundary of class $C^\infty$, $\bar\O_j \subset \O$ and $\bigcup_j \O_j=\O$ such that
$$
	\liminf_{k\to\infty} \int_{\O_j} h(x,u_k(x),\nabla u_k(x))\,\md x\geq \int_{\O_j} h(x,u(x),\nabla u(x))\,\md x.
$$
\end{lemma}
\begin{proof}
We select a subsequence of $\{u_k\}$ so that ``$\liminf=\lim$'' and  such that $\{(u_k)\}$ generates a Young measure $\nu$, and $\{(u_k,\nabla u_k)\}$ generates a measure $(\sigma,\hat\nu,\hat\mu)$ in the sense of \eqref{generate}. Now let $\O_0:=\emptyset$. For each $j$, we choose an open set $\O_j$ with smooth boundary such that
\[
  K_j:=\bar\O_{j-1}\cup \{x\in\O: \operatorname{dist}(x;\partial \O)\geq\tfrac{1}{j}\} \subset \O_j\subset \bar{\O}_j\subset \O
\]
and \vspace*{-3ex}
\begin{align}\label{noboundarycharge}
  \sigma(\partial \O_j)=0
\end{align}
Here, notice that since the distance of the compact set $K_j$ to $\partial\O$ is positive, we can find uncountably many pairwise disjoint candidates for $\O_j$. Since $\sigma$ is a finite measure, all but countably many of them must satisfy \eqref{noboundarycharge}.
Clearly, the measure generated by $\{(u_k,\nabla u_k)\}$ on $\O_j$ coincides with $(\sigma,\hat\nu,\hat\mu)$ on the open set $\O_j$,
and due to \eqref{noboundarycharge} even on $\bar\O_j$. 
Hence, by Theorem~\ref{thm:repgrad}, Remark~\ref{rem:uniform2}, 
\begin{align*}
\lim_{k\to\infty}  \int_{\O_j}h(x,u_k(x),\nabla u_k(x))\,\md x 
&=\int_{\O_j}\int_{\R^{m\times n}} h(x,u(x),s)\nu_x(\md s)\,\md x \\
&\quad+\int_{\bar\O_j}\int_{\b_\mathcal{R}\R^{m\times n}\setminus\R^{m\times n}}h_0^{(1)}(x,u(x),s)\hat\nu_x(\md s)\sigma(\md x)\\
&\geq \int_{\O_j} h(x,u(x),\nabla u(x))\,\md x.
\end{align*}
Here, the inequality above is due to Remark~\ref{rem:qc} and \eqref{rem6} with $\psi_0(s):=h_0^{(1)}(x,u(x),s)$ (separately applied for each $x$); for $\psi(s):=(1+|s|^p)\psi_0(s)$ and its quasiconvex hull $Q\psi$ we have $Q\psi>-\infty$ because $h(x,u(x),\cdot)$ is quasiconvex 
and $\psi(s)-h(x,u(x),s)=h_0^{(2)}(x,u(x),s)(1+|u(x)|^q)$ is bounded. 
\end{proof}

To get lower semicontinuity for all sequences and on the whole domain, we need an extra condition on the integrand on the boundary, namely, $p$-quasisubcritical growth from below, as in the case of integrands without explicit dependence on $u$ (cf.~Theorem~\ref{kroemer}).
\begin{theorem}\label{thm:p-greater-n}
Let $\O\subset\R^n$ be a bounded domain with a boundary of class $C^1$, let $\infty>p>n$ and let 
$h\in \HH^{q,p}(\O,\mathcal{U},\mathcal{R})$ (cf.~\eqref{defHpq}). 
Then, if $h(x,r,\cdot)$ is  quasiconvex for a.e.~$x\in\O$ (and therefore all $x\in\bar\O$, by continuity) and all $r\in\R^m$ and $\tilde h(x,s):=h(x,u(x),s)$ has $p$-quasisubcritical growth from below (see \eqref{pqslb}) for all $x\in\partial\O$ and all $u\in W^{1,p}(\O;\R^m)$,
$w\mapsto \int_\O h(x,w(x),\nabla w(x))\,\md x $ is weakly lower semicontinuous in $W^{1,p}(\O;\R^m)$.
\end{theorem}
\begin{proof}
Let $u_k\rightharpoonup u$ weakly in $W^{1,p}(\O;\R^m)$.
In view of Remark~\ref{rem:uniform2}, the measures generated by (subsequences of) $\{(u,\nabla u_k)\}$ and $\{(u_k,\nabla u_k)\}$ in the sense of \eqref{generate} always coincide. As a consequence of \eqref{generateall3} and \eqref{limit21},
it therefore suffices to show that for each $u\in W^{1,p}(\O;\R^m)\subset C(\bar\O;\R^m)$, $w\mapsto \int_\O h(x,u(x),\nabla w(x))\,\md x $ is weakly lower semicontinuous.
The latter follows from Theorem~\ref{thm:repgrad}.
%Here, notice that if $(x,s)\mapsto h(x,u(x),s)$ satisfies \eqref{pqslb}, so does its "$p$-growth-part" 
%$h_0^1(x,u(x),s)(1+|s|^p)$ as introduced in \eqref{defHpq}, because the rest $h_0^2(x,u(x),s)(1+|u(x)|^r)$ is bounded in $s$ and therefore  does not influence \eqref{pqslb}, which is a kind of growth condition in $s$ on the negative part of $h(x,u(x),s)$.
\end{proof}
\begin{remark}
In Theorem~\ref{thm:p-greater-n}, quasiconvexity of $h(x,u(x),\cdot)$ in $\O$ and $p$-qscb of $h(x,u(x),\cdot)$ at every $x\in \partial\O$ are also necessary for weak lower semicontinuity. We omit the details.
\end{remark}
As already briefly pointed out in the introduction, the situation becomes significantly more complicated if $p\leq n$. Using our measures to express the limit as in Theorem~\ref{thm:repgrad}, we can at least reduce the problem to a property of an integrand without explicit dependence on $u$, for each given sequence:
\begin{proposition}\label{prop:plessn}
Let $p\leq n$, suppose that $h(x,r,\cdot)$ is quasiconvex, $h\in \HH^{q,p}(\O,\mathcal{U},\mathcal{R})$, and let
$\{u_k\}\subset W^{1,p}(\O;\R^m)$
be a bounded sequence such that 
$u_k\rightharpoonup u$ and $\{(u_k,\nabla u_k)\}$ generates a measure $(\sigma,\hat\nu,\hat\mu)$ in the sense of \eqref{generate}.
Then 
$$
  \liminf_{k\to\infty} \int_\O h(x,u_k,\nabla u_k)\,\md x\geq \int_\O h(x,u,\nabla u)\,\md x,
$$
provided that for $\sigma$-a.e.~$x\in \bar\O$,
\begin{align}\label{plessn-cond}
  \int_{\bar \O}\int_{\beta_{\mathcal R}\R^{m\times n}\setminus \R^{m\times n}}\tilde{h}(x,s)\,\hat\nu_x(\md s)\sigma(\md x)\geq 0,	
\end{align}
where $\tilde{h}(x,s):=(1+|s|^p)\int_{\beta_{\mathcal U}} h_0^{(1)}(x,r,s) \,\hat\mu_{x,s}(\md r)$.
Here, recall that $h(x,r,s)=h_0^{(1)}(x,r,s)(1+|s|^p)+h_0^{(2)}(x,r,s)(1+|r|^q)$, cf.~\eqref{defHpq}.
\end{proposition}
\begin{proof}
This is a straightforward consequence of Theorem~\ref{thm:repgrad} and Remark~\ref{rem:qc}. 
\end{proof}
\begin{remark}
Given $h\in \HH^{q,p}(\O,\mathcal{U},\mathcal{R})$, 
$h_0^{(1)}(x,r,s)$ is uniquely determined for $s\in \beta_{\mathcal R}\R^{m\times n}\setminus \R^{m\times n}$, but not for 
$s\in \R^{m\times n}$. 
Of course,  \eqref{plessn-cond} actually is only a condition on the restriction of 
$h_0^{(1)}$ to $\bar\O\times \beta_{\mathcal U}\R^m\times (\beta_{\mathcal R}\R^{m\times n}\setminus \R^{m\times n})$.
\end{remark}

\section{Concluding remarks}
We have seen that  generalized DiPerna-Majda measures introduced here can be helpful in proofs of weak lower semicontinuity. Other applications are, for example, in impulsive control problems where the concentration of controls typically results in discontinuity of the state variable \cite{dhmktw}.  An open challenging problem is to find some explicit characterization of generalized Diperna-Majda measures generated by pairs of functions and their gradients, namely  $\{(u_k,\nabla u_k)\}\subset W^{1,p}(\O;\R^m)\times L^p(\O;\R^{m\times n})$. This could then help us to find necessary and sufficient conditions for weak lower semicontinuity of $u\mapsto \int_\O h(x, u(x),\nabla u(x))\,\md x$ in $W^{1,p}(\O;\R^m)$ for $1<p<+\infty$ and  for $h\in \mathbb{H}^p$. 

\bigskip
\begin{acknowledgement} This work was partly   done during MK's visiting Giovanni-Prodi professorship at the University of W\"{u}rzburg, Germany.  The hospitality and support of the Institute of Mathematics is gratefully acknowledged. This work was also supported by GA\v{C}R through projects 16-34894L and 17-04301S.
\end{acknowledgement}


\vspace*{1cm}

\bigskip










\bigskip

\begin{thebibliography}{19}
\baselineskip=12pt
{\footnotesize

\bibitem{af}
 Acerbi, E., Fusco, N. (1984) Semicontinuity problems in the
calculus of variations,  Arch. Rational Mech. Anal. 86:125--145.


\bibitem{ab}
 Alibert, J, Bouchitt\'{e}, G (1997) Non-uniform integrability and generalized Young measures  J. Convex Anal. 4:125--145.

%\bibitem{al}
%{\sc Allard, W.}: On the first variation of a varifold. {\it Ann.
%Math.} {\bf 95} (1972), 417-491.
%
%\bibitem{alm}
%{\sc Almgren, Jr., F. J.}: Existence and regularity almost
%everywhere of solutions to elliptic variational problems among
%surfaces of varying topological type and singularity structure.
%{\it Ann. Math.} {\bf 87} (1968), 321-391.

\bibitem{BKK16}
Ba\'ia, M., Kr\"omer, S., Kru\v{z}\'{\i}k, M. (2016) 
Generalized $\mathbf{W^{1,1}}$-Young measures and relaxation of problems with linear growth. Preprint arXiv:1611.04160v1, submitted.

\bibitem{ball3}
 Ball, J M (1989) A version of the fundamental theorem for Young measures. In:
{\it PDEs and Continuum Models of Phase Transition.} (Eds. M.Rascle, D.Serre,
M.Slemrod.) Lecture Notes in Physics  344, Springer, Berlin, 1989,
pp.207--215.

%\bibitem{murat}
%Ball, J.M., Murat, F. (1984) $W^{1,p}$-quasiconvexity and variational problems for multiple integrals. J. Funct. Anal.  58:225--253.

\bibitem{ball-zhang}
 Ball, J.M., Zhang K.-W. 1990 Lower semicontinuity of multiple integrals and the biting lemma.  Proc. Roy. Soc. Edinburgh  114A:67--379.

\bibitem{benesova-kruzik-SIREV}
Bene\v{s}ov\'{a}, B., Kru\v{z}\'{\i}k, M. (2017) Weak lower semicontinuity of integral functionals and applications. To appear in SIAM Review,  Preprint arxiv:1601.00390
%\bibitem{billingsley}
%{\sc Billingsley, P.}: {\it  Probability and Measure.} 3rd ed., John Wiley $\&$ Sons Ltd., Chichester, 1995.

\bibitem{c-h-k}
Claeys, M., Henrion, D., Kru\v{z}\'{\i}k, M. (2017) Semi-definite relaxations for optimal control problems with oscillations and concentration effects.  ESAIM Control Optim. Calc. Var.  23:95--117.

\bibitem{dacorogna}
 Dacorogna, B (2008) {\it  Direct Methods in the Calculus of Variations.} 2nd ed.,
Springer, Berlin.

\bibitem{diperna-majda}
 DiPerna, R.J., Majda, A.J. (1987) Oscillations and concentrations
in weak solutions of the incompressible fluid equations.
Commun.\ Math.\ Phys. 108:667--689.

%\bibitem{duch}  {\sc   Duhoux, M.}:
% Nonlinear singular Sturm--Liouville problems. {\it  Nonlinear
%Anal.} {\it 38} (1999), 897--918.


\bibitem{d-s}
 Dunford, N., Schwartz, J.T. (1967) {\it Linear Operators.}, Part I,
Interscience, New York


\bibitem{engelking}
 Engelking, R. (1985) {\it General topology.} 2nd ed., PWN, Warszawa.

\bibitem{evans0}
 Evans, L.C. (1990) {\it Weak Convergence Methods for Nonlinear Partial Differential Equations.} AMS Providence.
%
%\bibitem{evans}
%{\sc Evans, L.C., Gariepy, R.F.}: {\it Measure Theory and Fine Properties of Functions.} CRC Press, Inc. Boca Raton, 1992.
%
%\bibitem{fonseca}
%{\sc Fonseca, I.}: Lower semicontinuity of surface energies. {\it Proc. Roy. Soc. Edinburgh} {\bf 120A} (1992), 95--115.
%
%\bibitem{fonseca-gangbo}
%{\sc Fonseca, I., Gangbo, W.}: {\it Degree Theory in Analysis and Applications.} Oxford University Press, Oxford,  1995.
\bibitem{fmp}
 Fonseca, I., M\"{u}ller, S., Pedregal, P. (1998) Analysis of concentration and oscillation effects generated by gradients.  SIAM J. Math. Anal. 29:736--756.
%
%\bibitem{gis}
%{\sc Greco, L., Iwaniec, T., Subramanian, U.}: Another approach to
%biting convergence of Jacobians. {\it Illin. Journ. Math.} {\bf
%47}, No. 3 (2003), 815--830.
%
%\bibitem{guzman}
%{\sc de Guzm\'an, M.}: {\it Differentiation of integrals in $\R^n$},
  %Lecture Notes in Math. {\bf 481}, Springer, Berlin,  1975.
%
%\bibitem{hafsa}
%{\sc Hafsa, O.M., Mandallena, J.-P., Michaille, G.}: Homogenization of periodic nonconvex integral  functionals in terms of Young measures. {\it ESAIM:COCV} {\bf 12} (2006), 35--51.
%

%\bibitem{ak1}
%{\sc Ka\l{}amajska, A.} On lower semicontinuity of multiple
%integrals, {\it Coll. Math.} {\bf 74} (1997), 71--78.
%



\bibitem{dhmktw}
 Henrion, D., Kru\v{z}\'{\i}k, M., Weisser, T. (2017) Optimal control problems with oscillations, concentrations, and discontinuities. In preparation.

\bibitem{ak2}
{Ka\l{}amajska, A.} On Young measures controlling
discontinuous functions, {\it J. Conv. Anal.} {\bf 13} (2006),
No.1,  177--192.

\bibitem{mkak}
 Ka\l{}amajska, A., Kru\v{z}\'{\i}k, M. (2008) Oscillations and concentrations in sequences of gradients.   ESAIM Control Optim. Calc. Var.  14:71--104.

\bibitem{k-p1}
Kinderlehrer, D., Pedregal, P. (1991) Characterization of Young measures generated by gradients.  Arch. Rational Mech. Anal.  115:329--365.

\bibitem{k-p}
Kinderlehrer, D., Pedregal, P. (1994) Gradient Young measures
generated by sequences in Sobolev spaces.  J. Geom. Anal. 4:59--90.

\bibitem{KriRin_YM_10}
Kristensen J.,  Rindler F. (2010), and Erratum (2012) Characterization of generalized gradient Young measures generated by sequences in $W^{1,1}$ and $BV$ {\it Arch. Rat. Mech. Anal.} {\bf 197}, 539--598, and
{\bf 203}, 693--700.

\bibitem{kroemer}
Kr\"{o}mer, S. (2010) On the role of lower bounds in characterizations of weak lower semicontinuity of multiple integrals. Adv. Calc. Var. 3:387--408.

\bibitem{kroemer-kruzik}
Kr\"{o}mer, S.,  Kru\v{z}\'{\i}k, M. (2013) Oscillations and concentrations in sequences of gradients up to the boundary. J. Convex Anal.  20:723--752. 

%\bibitem{kristensen}
%{\sc Kristensen, J.}: Lower semicontinuity in spaces of weakly differentiable functions. {\it Math. Ann.} {\bf  313} (1999), 653-710.

\bibitem{k-r-dm}
 Kru\v{z}\'{\i}k, M., Roub\'{\i}\v{c}ek, T. (1997)
On the measures of DiPerna and Majda.   Mathematica Bohemica 122:383--399.

\bibitem{k-r-control}
Kru\v{z}\'{\i}k, M., Roub\'{\i}\v{c}ek, T. (1999) Optimization problems with
concentration and oscillation effects: relaxation theory and numerical
approximation. Numer. Funct. Anal. Optim. 20:511--530.

\bibitem{licht}
Licht, C., Michaille, G., Pagano, S. (2007) A model of elastic adhesive bonded joints through oscillation-concentration measures. J. Math. Pures Appl. 87:343–-365.





\bibitem{meyers}
Meyers, N.G. (1965) Quasi-convexity and lower semicontinuity of
multiple integrals of any order.  Trans. Am. Math. Soc.
119:125--149.

\bibitem{morrey}
Morrey, C.B. (1966) {\it Multiple Integrals in the Calculus of Variations.} Springer, Berlin.


%\bibitem{mueller}
%{\sc M\"{u}ller, S.}: {\it Variational models for microstructure and phase transisions.} Lecture Notes in Mathematics {\bf 1713} (1999) pp. 85--210.

\bibitem{paroni1}
Paroni, R., Tomassetti, G. (2009) A variational justification of linear elasticity with residual stress. J. Elasticity 97:189--206.

\bibitem{paroni2}
Paroni, R., Tomassetti, G. (2011) From non-linear elasticity to linear elasticity with initial stress via $\Gamma$-convergence.
Cont. Mech. Thermodyn. 23:347--361.
	
\bibitem{pedregal}
 Pedregal, P. (1997) {\it Parametrized Measures and Variational Principles.}
Birk\"auser, Basel.

\bibitem{pedregal-mult}
Pedregal, P (2005) Multiscale Young measures. Trans. Am. Math. Soc. 358:591--602.

%\bibitem{reshetnyak}
%{\sc Reshetnyak, Yu.G.}: The generalized derivatives and
%the a.e. differentiability, {\it Mat. Sb.} {\bf 75} (1968), 323--334.
%(in Russian).
%
%\bibitem{re}
%{\sc Reshetnyak, Yu.G.}: Weak convergence and completely additive
%vector functions on a set. {\it Sibirsk. Mat. Zh.} {\bf 9} (1968),
%1039-1045.

\bibitem{r}
Roub\'{\i}\v{c}ek, T. (1997) {\it Relaxation in Optimization Theory and
Variational Calculus.} W. de Gruyter, Berlin.

\bibitem{schonbek}
 Schonbek, M.E. (1982) Convergence of solutions to nonlinear dispersive
equations. Comm. in Partial Diff. Equations 7:959--1000.

\bibitem{warga}
 Warga, J. (1972) {\it Optimal Control of Differential and Functional
Equations.} Academic Press, New York.

\bibitem{y}
Young, L.C. (1937) Generalized curves and the existence of an attained
absolute minimum in the calculus of variations.  Comptes Rendus de la
Soci\'et\'e des Sciences et des Lettres de Varsovie, Classe III  30:212--234.

}

\end{thebibliography}
\end{document}